\documentclass[11pt]{article}%
\usepackage{amsmath}
\usepackage{amssymb}
\usepackage{amsfonts}
\usepackage{graphicx}
\usepackage{enumitem}
\usepackage{hyperref}
\usepackage{color}

\usepackage{tikz}
\usetikzlibrary{shapes.misc, arrows, fit, calc}

\providecommand{\U}[1]{\protect\rule{.1in}{.1in}}

\allowdisplaybreaks
\newtheorem{theorem}{Theorem}[section]

\newtheorem{definition}[theorem]{Definition}

\newtheorem{proposition}[theorem]{Proposition}
\newtheorem{corollary}[theorem]{Corollary}
\newtheorem{lemma}[theorem]{Lemma}

\newtheorem{remark}[theorem]{Remark}

\newtheorem{example}[theorem]{Example}

\newtheorem{algorithm}[theorem]{Algorithm}
\newenvironment{proof}{\noindent{\em Proof:}}{$\Box$~\\}

\newcommand*{\rrec}[1]{\tikz[baseline=(char.base), text depth=0pt]\node[anchor=south west, draw,rectangle, rounded corners=10pt, inner sep=2pt, minimum size=7mm, text height=2mm](char){\ensuremath{#1}} ;}
\newcommand{\Pone}{\operatorname{P}_1}
\newcommand{\Ptwo}{\operatorname{P}_2}
\newcommand{\Pl}{\operatorname{P}_{\ell}}
\newcommand{\E}{\operatorname{E}}

\begin{document}

\title{Performance Analysis of the Solving Algorithm for the Kuramoto Model with Rank One Coupling}
\author{Owen Coss\thanks{Department of Mathematics, North Carolina State University (otcoss@gmail.com)}}
\date{\today}
\maketitle

\begin{abstract}
\noindent This paper is a follow up to a previous work that presented an algorithm to efficiently find all of the equilibria of the Kuramoto model with nonuniform coupling described by a rank one matrix. The algorithm was shown experimentally to be more efficient than previously used methods, but its performance was not fully characterized. This paper analyzes the effectiveness of the ``pruning" method used to skip cases with no solutions. The approach utilized is to construct a weighted graph where every path through the graph corresponds to the algorithm's performance on an input. The maximum weight path then corresponds to the worst case performance of the algorithm. This paper shows that even in the worst case, the pruning method employed is very effective at skipping cases with no solutions.

\medskip

\end{abstract}

This paper is a follow up to a previous work that presented a solving algorithm for the Kuramoto model \cite{kuramoto1975} with rank one coupling \cite{coss2018}. In that paper, the algorithm was experimentally shown to be more efficient than other solving methods, but was not fully characterized. This paper examines the ``pruning" method that was used to skip cases and shows that it efficiently skips cases without solutions. The structure of this paper is as follows. The introduction has a brief summary of the two algorithms from \cite{coss2018}. Section one explains how the algorithms can be converted into a weighted graph. Sections two and three then examine the maximum weight paths for the two algorithms and show the worst case performances are $n\vert S \vert$ and $\vert S \vert$ extraneous cases checked where $n$ is the number of oscillators and $S$ is the set of cases with solutions.

The essence of the algorithms from \cite{coss2018} involve the idea of pruning out cases that have no solutions. To do so, a pruning function is used.
\begin{definition}
	The \textbf{pruning function} $\Pl: \mathbb{Z}_{>0} \rightarrow \mathbb{Z}_{\geq 0}$ for $\ell \in \mathbb{Z}_{>0}$ takes the binary representation of $m$, zeroes out everything to the right of the $\ell^{\text{th}}$ zero from the right (left padding with zeros as needed), and then subtracts one. If that number is negative, then it returns 0, otherwise it returns that number.
\end{definition}

\begin{example} Note that $23 = 10111_2$. Then
	\begin{align*}
	\Pone(23) & = 01111_2 = 15 & \Ptwo(23) & = 0
	\end{align*}
\end{example}

The cases to be checked can be numbered the from 0 to $2^{n}-1$. The pruning algorithm starts with case $2^{n}-1$ and works its way down. If a solution for a case is found, then the next case is checked. However, if the case has no solutions, then a number of subsequent cases are also known to have no solutions and can be pruned. In the case of the rank one coupled Kuramoto model, $\Pone$ can always be employed, but in special circumstances, $\Ptwo$ can be used instead to more aggressively prune out extraneous cases. A simplified representation of this algorithm is as follows.

\begin{algorithm}[Prune $\ell$]\label{prune_alg}\
	
	\begin{enumerate}
		\item $j \leftarrow 2^{n} - 1$
		
		\item While $j > 0$ do
		\begin{enumerate}
			\item If case $j$ has a solution, $j \leftarrow j-1$
			\item Else, $j \leftarrow \Pl(j)$
		\end{enumerate}
	\end{enumerate}
\end{algorithm}

\begin{remark}
	Note that the case associated with 0 never has a solution, so it can always be skipped.
\end{remark}

\section{Graph Setup}

This section shows how Algorithm \ref{prune_alg} can be converted into a weighted graph. The first step is to adjust the algorithm to count how many times it checked a case and how many times it found a case with a solution. Throughout the rest of the paper, $S$ will be the set of cases that have solutions. We will only consider situations that have at least one solution, so $S \neq \emptyset$.
\begin{definition}\label{def-rk1-efficiency_alg}
	The \textbf{efficiency algorithm} $\E_{\ell}$ is the algorithm that takes as its input a nonempty subset of $\{1,2,\ldots,2^n-1 \}$ for some given integers $n \geq 2$ and $\ell > 0$ and outputs a pair of integers $(R,C)$ according to the following steps:
	\begin{description}[leftmargin=3em,style=nextline,itemsep=0.5em]
		
		\item[In:] $S \subseteq \{1,2,\ldots,2^n-1\}~$ s.t. $~S \neq \emptyset$
		
		\item[Out:] $(R,C)$
	\end{description}
	\begin{enumerate}
		\item $j \leftarrow 2^n - 1,\; R \leftarrow 0,\; C \leftarrow 0$
		
		\item While $j > 0$ do
		\begin{enumerate}
			\item $C \leftarrow C + 1$
			
			\item If $j \in S$ then
			\begin{enumerate}
				\item $R \leftarrow R+1$
				
				\item $j \leftarrow j - 1$
			\end{enumerate}
			
			\item Else
			\begin{enumerate}
				\item $j \leftarrow \Pl(j)$
			\end{enumerate}
		\end{enumerate}
	\end{enumerate}
	
	For a given $\ell$ and $n$, an input $S \subseteq \{1,2,\ldots,2^n-1 \},~ S \neq \emptyset$ is called $\textbf{valid}$ if $R = \vert S \vert$ where $(R,C)$ is the output of $\E_{\ell}(S)$. Let $S_{\ell}$ be the set of all valid inputs.
\end{definition}
The outputs of the efficiency algorithm $\E_{\ell}$ on a valid input $S$ can be interpreted as 
\begin{itemize}
	\item $R = $ the number of cases checked that are required to be checked
	\item $C = $ the total number of cases that are checked
\end{itemize}
Comparing the sizes of $R$ and $C$ quantifies the efficiency of the pruning algorithm $\Pl$ since the closer $C$ is to $R$, the less work was wasted checking unnecessary cases. We restrict the input to valid sets since we will assume that the algorithm we are abstracting is correct, i.e., it will not skip over any case that does have a solution. The following Lemma presents an immediate fact about valid inputs.

\begin{lemma}\label{lem-rk1-all_pos_in_Sl}
	$2^{n}-1 \in S$ for any $S \in S_{\ell}$.
\end{lemma}
\begin{proof}
	Note that the $n$ digit binary representation of $2^{n}-1$ is all ones. Therefore $\Pl(2^{n}-1) = 0$ for any $\ell > 0$. Suppose that $2^{n}-1 \notin S$ for some $S \in S_{\ell}$. Then $\E_{\ell}(S) = (0, 1)$, but $\vert S \vert > 0$ by definition of $S_{\ell}$. Since this is a contradiction, we must have $2^{n}-1 \in S$.
\end{proof}

In order to consider the algorithmic complexity, we will build a graph that represents all the possible outputs of $\E_{\ell}$ as all the possible paths through this graph. Building all possible $(R,\; C)$ pairs produced by $\E_{\ell}(S_{\ell})$ can be done recursively by working backwards. Let $\E_{\ell}^{m}$ be the same as $\E_{\ell}$ except that the algorithm starts ``in the middle" by initializing $j$ to $m$ in Step 1 of Definition \ref{def-rk1-efficiency_alg}. For a given $S \in S_{\ell}$, at any case $m < 2^{n}-1$, there are two options:
\begin{itemize}
	\item $m \in S$, in which case we get $\E_{\ell}^{m}(S) = (\alpha+1,\; \beta+1)$ where $(\alpha,\; \beta)$ is the result of $\E_{\ell}^{m-1}(S)$.
	
	\item $m \notin S$, in which case we get $\E_{\ell}^{m}(S) = (\alpha,\; \beta+1)$ where $(\alpha,\; \beta)$ is the result of $\E_{\ell}^{\Pl(m)}(S)$.
\end{itemize}
For $m = 2^{n} - 1$, we have only the first option by Lemma \ref{lem-rk1-all_pos_in_Sl}.

To determine the asymptotic complexity, we want to show that $\displaystyle\underset{(R,C) \in \E_{\ell}(S_{\ell})}{\max}\; \frac{C}{R} = f(n)$. To do so, we can rewrite any $(R,C)$ pair as $(\alpha,\; f(n)\cdot\alpha+p)$ where $p$ is the ``par number" of the pair. If $p \leq 0$ for every pair, then the ratio is at most $f(n)$, and if $p=0$ for some pair, then the maximum is exactly $f(n)$. Consider the two cases above again. 
\begin{itemize}
	\item $m \in S$, so we have $\E_{\ell}^{m}(S) = (\alpha+1,\; f(n)\cdot(\alpha+1)+ (p - f(n)+1))$ where $(\alpha,\; f(n)\cdot \alpha + p )$ is the result of $\E_{\ell}^{m-1}(S)$. Therefore the new par number is lower by $f(n)-1$.
	
	\item $m \notin S$, so we get that $\E_{\ell}^{m}(S) = (\alpha,\; f(n)\cdot\alpha+(p+1))$ where $(\alpha,\; f(n)\cdot \alpha + p)$ is the result of $\E_{\ell}^{\Pl(m)}(S)$. Therefore the new par number is one higher.
\end{itemize} 

To represent this graphically, we can connect the node \rrec{m-1} to \rrec{m} with an edge having weight $-f(n)+1$ for $m \leq 2^{n}-1$, and we can connect the node \rrec{\Pl(m)} to \rrec{m} with an edge having weight $+1$ for $m < 2^{n}-1$. Showing that $\displaystyle\underset{(R,C) \in \E_{\ell}(S_{\ell})}{\max}\; \frac{C}{R} = f(n)$ can now be done by showing that the maximum weight path through this graph is zero.

Before beginning to examine the efficiency of $\Pone$ and $\Ptwo$, several lemmas regarding the binary representation of sign cases are required.

\begin{lemma} \label{leading_one}
	The $k \geq 1$ digit binary representation of $2^{k-1} \leq m < 2^{k}$ has a leading one.
\end{lemma}
\begin{proof}
	$2^{k-1}$ has the binary representation of a one followed by $k-1$ zeros. $2^{k}-1$ has the binary representation of $k$ ones. Thus every number in between must start with a one and have some combination of $k-1$ ones and zeros.
\end{proof}

\begin{lemma} \label{leading_10}
	The $k \geq 2$ digit binary representation of $2^{k-1} \leq m \leq 2^{k} - 2^{k-2} - 1$ starts with ``10."
\end{lemma}
\begin{proof}
	$2^{k-1}$ has the binary representation of a one followed by $k-1$ zeros. $2^{k} - 2^{k-2} - 1$ has the binary representation of a ``10" followed by $k-2$ ones. Thus every number in between starts with ``10" and has some combination of $k-2$ ones and zeros.
\end{proof}

\begin{lemma} \label{leading_101}
	The $k \geq 3$ digit binary representation of $2^{k} - 2^{k-2} - 2^{k-3} \leq m \leq 2^{k} - 2^{k-2} - 1$ starts with ``101."
\end{lemma}
\begin{proof}
	$2^{k} - 2^{k-2} - 2^{k-3}$ has a binary representation that starts with ``101" followed by $k-3$ zeros. $2^{k} - 2^{k-2} - 1$ has a binary representation that starts with ``101" followed by $k-3$ ones. Thus every number in between starts with ``101" and has some combination of $k-3$ ones and zeros.
\end{proof}

\begin{lemma} \label{leading_100}
	The $k \geq 3$ digit binary representation of $2^{k-1} \leq m \leq 2^{k} - 2^{k-2} - 2^{k-3} - 1$ starts with ``100."
\end{lemma}
\begin{proof}
	$2^{k-1}$ has a binary representation of a one followed by $k-1$ zeros. $2^{k} - 2^{k-2} - 2^{k-3} - 1$ has a binary representation that starts with ``100" and is followed by $k-3$ ones. Thus every number in between starts with ``100" and has some combination of $k-3$ ones and zeros.
\end{proof}

\begin{example}
	Let $k = 4$.
	\begin{itemize}
		\item Lemma \ref{leading_one}: Every $8 \leq m < 16$ starts with a leading one
		\[
		8 = 1000_2,~ 9 = 1001_2,~ 10 = 1010_2, \ldots, 15 = 1111_2
		\]
		
		\item Lemma \ref{leading_10}: Every $8 \leq m \leq 11$ starts with ``10."
		\[
		8 = 1000_2,~ 9 = 1001_2,~ 10 = 1010_2,~ 11 = 1011_2
		\]
		
		\item Lemma \ref{leading_101}: Every $10 \leq m \leq 11$ starts with ``101."
		\[
		10 = 1010_2,~ 11 = 1011_2
		\]
		
		\item Lemma \ref{leading_100}: Every $8 \leq m \leq 9$ starts with ``100."
		\[
		8 = 1000_2,~ 9 = 1001_2
		\]
	\end{itemize}
\end{example}

\section{Prune 1}
We will show that $\displaystyle\underset{(R,C) \in \E_1(S_1)}{\max}\; \frac{C}{R} = n+1$.

\begin{example}
	Let $n = 3$. Then the graph representing $\E_{1}$ is
	
	
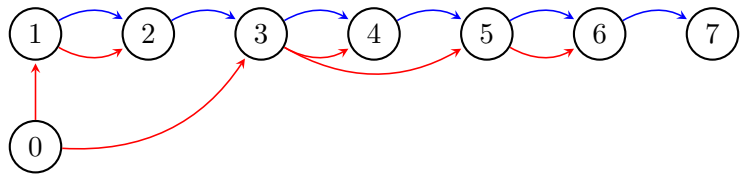
\vspace*{0.5cm}
	{\centering
		\begin{tikzpicture}[> = stealth, shorten > = 1pt, semithick, node distance=1.5cm]
		\tikzstyle{crc} = [draw=black, thick, circle, fill=white]
		
		\node[crc](7){$7$};
		\node[crc](6)[left of=7]{$6$};
		\node[crc](5)[left of=6]{$5$};
		\node[crc](4)[left of=5]{$4$};
		\node[crc](3)[left of=4]{$3$};
		\node[crc](2)[left of=3]{$2$};
		\node[crc](1)[left of=2]{$1$};
		\node[crc](0)[below of=1]{$0$};
		
		\path[->, bend left, blue] (1) edge (2);
		\path[->, bend left, blue] (2) edge (3);
		\path[->, bend left, blue] (3) edge (4);
		\path[->, bend left, blue] (4) edge (5);
		\path[->, bend left, blue] (5) edge (6);
		\path[->, bend left, blue] (6) edge (7);
		
		\path[->, bend right, red] (1) edge (2);
		\path[->, bend right, red] (0) edge (3);
		\path[->, bend right, red] (3) edge (4);
		\path[->, bend right, red] (3) edge (5);
		\path[->, bend right, red] (5) edge (6);
		\path[->, red] (0) edge (1);
		\end{tikzpicture}
		
	}
	\noindent where blue lines have weight $-3$ and red lines have weight $+1$. For $S = \{6,7 \} \in S_{1}$, $\E_{1}(S)$ would start with $j = 7$, go to $6$, go to $5$, skip to $3$, then skip to $0$, returning $(2, 4) = (2, 4\cdot 2 - 4)$, so this particular example finishes four under par for the max ratio $C/R$ of $4$. We can convert the steps of $\E_{1}$ on $S$ into a path in the graph above by reversing the order of the $j$ values. Traveling from \rrec{0} to \rrec{3} on the red edge then to \rrec{5} on the red edge, then to \rrec{6} on the blue edge, then to \rrec{7} on the blue edge gives a total weight of $1 + 1 + -3 + -3 = -4$, the par number for $S$.	
\end{example}

It is useful to break the full graph for $\E_{1}$ into several parts. Consider the weighted directed graph $G_k$ for $k \geq 2$ with nodes \rrec{0}, \rrec{2^{k-1}-1}, \rrec{2^{k-1}}, \ldots, \rrec{2^{k}-1} with edges based on the following rules.
\begin{itemize}
	\item An edge with weight $-n$ goes from \rrec{m} to \rrec{m+1} for $2^{k-1}-1 \leq m < 2^{k}-1$.
	
	\item An edge with weight $+1$ goes from \rrec{\Pone(m)} to \rrec{m} for $2^{k-1} -1 \leq m < 2^{k}-1$.
\end{itemize}
Note that \rrec{0} will have no incoming edges, \rrec{2^{k-1}-1} will have one incoming edge of weight $+1$, \rrec{2^{k-1}},\ldots,\rrec{2^{k}-2} will each have two incoming edges, one of weight $-n$ and one of weight $+1$, and \rrec{2^{k}-1} will have one incoming edge of weight $-n$. There are no other edges. Also note that edges go from a smaller number to a larger. Joining $G_2$, $G_3$, \ldots, $G_n$ together and examining the maximum weight path gives the worst case performance for $\E_{1}$.

\begin{example}\label{P1:G_2 - G_4}
	$G_2$, $G_3$, and $G_4$. Blue edges have weight $-n$ and red edges have weight $+1$. The graph for $\E_{1}$ when $n=4$ is $G_2 \cup G_3 \cup G_4$.
	
	\vspace*{0.5cm}
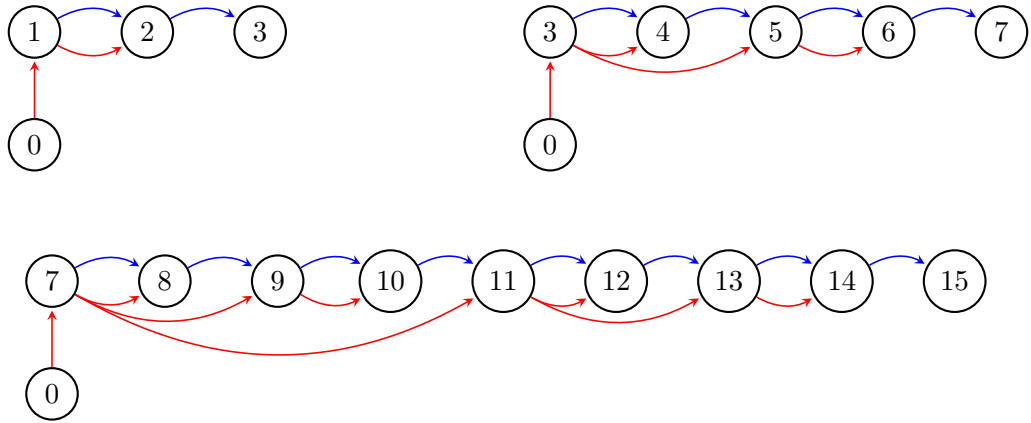
	\begin{minipage}{0.35\linewidth}
		\centering
		\begin{tikzpicture}[> = stealth, shorten > = 1pt, semithick, node distance=1.5cm]
		\tikzstyle{crc} = [draw=black, thick, circle, fill=white]
		
		\node[crc](3){$3$};
		\node[crc](2)[left of=3]{$2$};
		\node[crc](1)[left of=2]{$1$};
		\node[crc](0)[below of=1]{$0$};
		
		\path[->, bend left, blue] (1) edge (2);
		\path[->, bend left, blue] (2) edge (3);
		
		\path[->, bend right, red] (1) edge (2);
		\path[->, red] (0) edge (1);
		\end{tikzpicture}
	\end{minipage}
	\begin{minipage}{0.65\linewidth}
		\centering
		\begin{tikzpicture}[> = stealth, shorten > = 1pt, semithick, node distance=1.5cm]
		\tikzstyle{crc} = [draw=black, thick, circle, fill=white]
		
		\node[crc](7){$7$};
		\node[crc](6)[left of=7]{$6$};
		\node[crc](5)[left of=6]{$5$};
		\node[crc](4)[left of=5]{$4$};
		\node[crc](3)[left of=4]{$3$};
		\node[crc](0)[below of=3]{$0$};
		
		\path[->, bend left, blue] (3) edge (4);
		\path[->, bend left, blue] (4) edge (5);
		\path[->, bend left, blue] (5) edge (6);
		\path[->, bend left, blue] (6) edge (7);
		
		\path[->, bend right, red] (3) edge (4);
		\path[->, bend right, red] (3) edge (5);
		\path[->, bend right, red] (5) edge (6);
		\path[->, red] (0) edge (3);
		\end{tikzpicture}
	\end{minipage}
	
	\vspace*{1cm}
	\centering
	\begin{tikzpicture}[> = stealth, shorten > = 1pt, semithick, node distance=1.5cm]
	\tikzstyle{crc} = [draw=black, thick, circle, fill=white]
	
	\node[crc](15){$15$};
	\node[crc](14)[left of=15]{$14$};
	\node[crc](13)[left of=14]{$13$};
	\node[crc](12)[left of=13]{$12$};
	\node[crc](11)[left of=12]{$11$};
	\node[crc](10)[left of=11]{$10$};
	\node[crc](9)[left of=10]{$9$};
	\node[crc](8)[left of=9]{$8$};
	\node[crc](7)[left of=8]{$7$};
	\node[crc](0)[below of=7]{$0$};
	
	\path[->, bend left, blue] (7) edge (8);
	\path[->, bend left, blue] (8) edge (9);
	\path[->, bend left, blue] (9) edge (10);
	\path[->, bend left, blue] (10) edge (11);
	\path[->, bend left, blue] (11) edge (12);
	\path[->, bend left, blue] (12) edge (13);
	\path[->, bend left, blue] (13) edge (14);
	\path[->, bend left, blue] (14) edge (15);
	
	\path[->, bend right, red] (7) edge (8);
	\path[->, bend right, red] (7) edge (9);
	\path[->, bend right, red] (9) edge (10);
	\path[->, bend right, red] (7) edge (11);
	\path[->, bend right, red] (11) edge (12);
	\path[->, bend right, red] (13) edge (14);
	\path[->, bend right, red] (11) edge (13);
	\path[->, red] (0) edge (7);
	\end{tikzpicture}
\end{example}

\begin{example} \label{P1:G_5}
	$G_5$. Blue edges have weight $-n$ and red edges have weight $+1$.
	
	
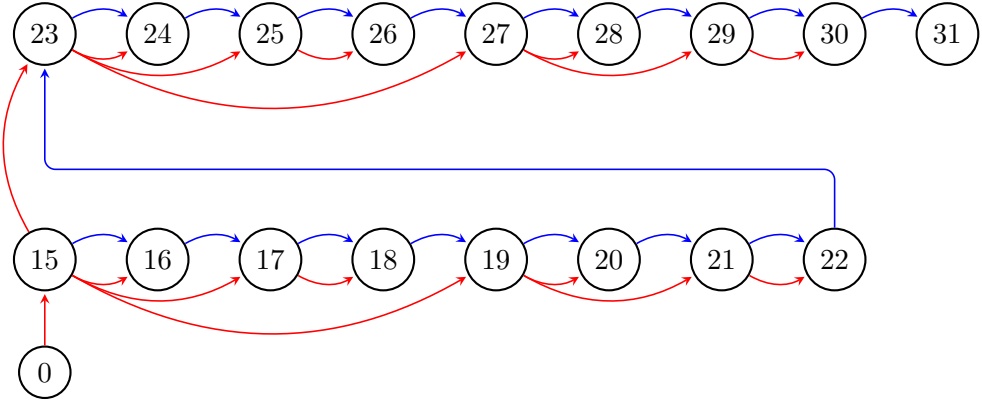
\vspace*{0.5cm}
	\centering
	\begin{tikzpicture}[> = stealth, shorten > = 1pt, semithick, node distance=1.5cm]
	\tikzstyle{crc} = [draw=black, thick, fill=white, circle]
	
	\node[crc](31){$31$};
	\foreach \x in {30,29,...,23}{
		\pgfmathtruncatemacro\y{\x+1}
		\node[crc](\x)[left of=\y]{$\x$};
	}
	
	\foreach \x in {23,...,30}{
		\pgfmathtruncatemacro\y{\x+1}
		\path[->, bend left, blue](\x) edge (\y);
	}
	
	\node[crc](22)[below of=30, node distance=3cm]{$22$};
	\foreach \x in {21,20,...,15}{
		\pgfmathtruncatemacro\y{\x+1}
		\node[crc](\x)[left of=\y]{$\x$};
	}
	
	\foreach \x in {15,...,21}{
		\pgfmathtruncatemacro\y{\x+1}
		\path[->, bend left, blue](\x) edge (\y);
	}
	
	\node[crc](0)[below of=15]{$0$};
	
	\draw[->, blue, rounded corners] (22) |- ($ (27)!0.6!(19) $) -| (23);
	
	\path[->, bend right, red] (29) edge (30);
	\path[->, bend right, red] (27) edge (29);
	\path[->, bend right, red] (27) edge (28);
	\path[->, bend right, red] (23) edge (27);
	\path[->, bend right, red] (25) edge (26);
	\path[->, bend right, red] (23) edge (25);
	\path[->, bend right, red] (23) edge (24);
	\path[->, bend left, red] (15) edge (23);
	\path[->, bend right, red] (21) edge (22);
	\path[->, bend right, red] (19) edge (21);
	\path[->, bend right, red] (19) edge (20);
	\path[->, bend right, red] (15) edge (19);
	\path[->, bend right, red] (17) edge (18);
	\path[->, bend right, red] (15) edge (17);
	\path[->, bend right, red] (15) edge (16);
	\path[->, red] (0) edge (15);
	\end{tikzpicture}
	
\end{example}

\begin{definition}
	Let
	\begin{align*}
	I^k_t   & = 2^{k}-1 		& II^k_t  & = 2^{k}-2^{k-2}-2\\
	I^k_b   & = 2^{k}-2^{k-2}-1 & II^k_b  & = 2^{k-1}-1
	\end{align*}
	Then $G_k$ for $k \geq 3$ can be split into two ``boxes" as follows:
	\begin{itemize}
		\item $B_{1}$: The induced subgraph of $G_k$ by taking \rrec{0}, \rrec{I^k_b}, \ldots, \rrec{I^k_t}.
		
		\item $B_{2}$: The induced subgraph of $G_k$ by taking \rrec{0}, \rrec{II^k_b}, \ldots, \rrec{II^k_t}.
	\end{itemize}
\end{definition}

\begin{example} \label{P1:G_4 boxes ex}
	Boxes of $G_4$. Blue edges have weight $-n$ and red edges have weight $+1$. Dashed edges go between boxes.
	
	
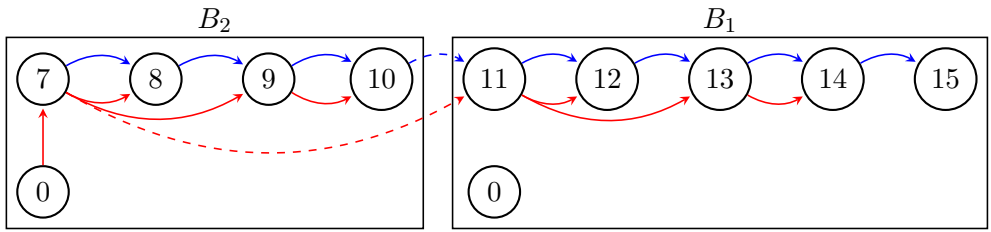
\vspace*{0.5cm}
	\centering
	\begin{tikzpicture}[> = stealth, shorten > = 1pt, semithick, node distance=1.5cm]
	\tikzstyle{crc} = [draw=black, thick, circle, fill=white]
	
	\node[crc](15){$15$};
	\node[crc](14)[left of=15]{$14$};
	\node[crc](13)[left of=14]{$13$};
	\node[crc](12)[left of=13]{$12$};
	\node[crc](11)[left of=12]{$11$};
	\node[crc](10)[left of=11]{$10$};
	\node[crc](9)[left of=10]{$9$};
	\node[crc](8)[left of=9]{$8$};
	\node[crc](7)[left of=8]{$7$};
	\node[crc](0I)[below of=11]{$0$};
	\node[crc](0II)[below of=7]{$0$};
	
	\node[draw, fit=(15) (11) (0I)](I){};
	\node[above of=I, node distance=1.5cm](I-label){$B_1$};
	\node[draw, fit=(10) (7) (0II)](II){};
	\node[above of=II, node distance=1.5cm](II-label){$B_2$};
	
	\path[->, bend left, blue] (7) edge (8);
	\path[->, bend left, blue] (8) edge (9);
	\path[->, bend left, blue] (9) edge (10);
	\path[->, bend left, blue, dashed] (10) edge (11);
	\path[->, bend left, blue] (11) edge (12);
	\path[->, bend left, blue] (12) edge (13);
	\path[->, bend left, blue] (13) edge (14);
	\path[->, bend left, blue] (14) edge (15);

	\path[->, red] (0II) edge (7);
	\path[->, bend right, red] (7) edge (8);
	\path[->, bend right, red] (7) edge (9);
	\path[->, bend right, red] (9) edge (10);
	\path[->, bend right, red, dashed] (7) edge (11);
	\path[->, bend right, red] (11) edge (12);
	\path[->, bend right, red] (11) edge (13);
	\path[->, bend right, red] (13) edge (14);
	\end{tikzpicture}
\end{example}

\paragraph*{Recursive Structure}
$G_{k+1}$ can be constructed from $G_{k}$ using the following propositions.

\begin{proposition} \label{P1:Box I structure}
	$B_1$ of $G_{k+1}$ for $k \geq 2$ is the same as $G_{k}$ with $2^{k}$ added to every nonzero node except that \rrec{I_{b}^{k+1}} has a blue edge from \rrec{II_{t}^{k+1}} and a red edge from \rrec{II_{b}^{k+1}} instead of \rrec{0}. 
\end{proposition}
\begin{proof}
	\begin{itemize}
		\item Nodes:
		
		It is sufficient to show that the top and bottom nodes of $G_k$ with $2^{k}$ added match the top and bottom nodes of $B_1$ of $G_{k+1}$ since all the nodes in between are sequential.
		\begin{align*}
		II_{b}^{k} + 2^{k} & = (2^{k-1}-1) + 2^{k}     & I_{t}^{k} + 2^{k} & = (2^{k} - 1) + 2^{k} \\
		& = 2^{k+1} - (2^{k+1} - 2^{k} - 2^{k-1}) - 1  & & = 2^{k+1} - 1 \\
		& = 2^{k+1} - 2^{k-1} - 1 					   & & = I_{t}^{k+1} \\
		& = I_{b}^{k+1}								   & &
		\end{align*}
		
		\item Blue Edges:
		
		\rrec{m+2^{k-1}} and \rrec{(m+1)+2^{k-1}} are still adjacent. Also \rrec{I_{b}^{k+1}} and \rrec{II_{t}^{k+1}} are adjacent.
		
		\item Red Edges:
		
		Note that for $2^{k-1} \leq m < I_{t}^{k}$ the $k$ digit binary representation of $m$ starts with a one by Lemma \ref{leading_one} and has at least one zero since $I_{t}^{k}$ consists of $k$ ones. Adding $2^{k}$ to $m$ left-appends a one to the binary representation. Therefore $\Pone(m + 2^{k}) = \Pone(m) + 2^{k}$ since the left-most one will be unaffected by $\Pone$. Furthermore, $I_{b}^{k+1}$ has the binary representation of ``10" followed by $k-1$ ones. Therefore $\Pone(I_{b}^{k+1}) = II_{b}^{k+1}$.
	\end{itemize}
\end{proof}
\begin{example} \label{P1:Box I structure ex}
	$B_1$ of $G_4$ and $G_3$.
	
	
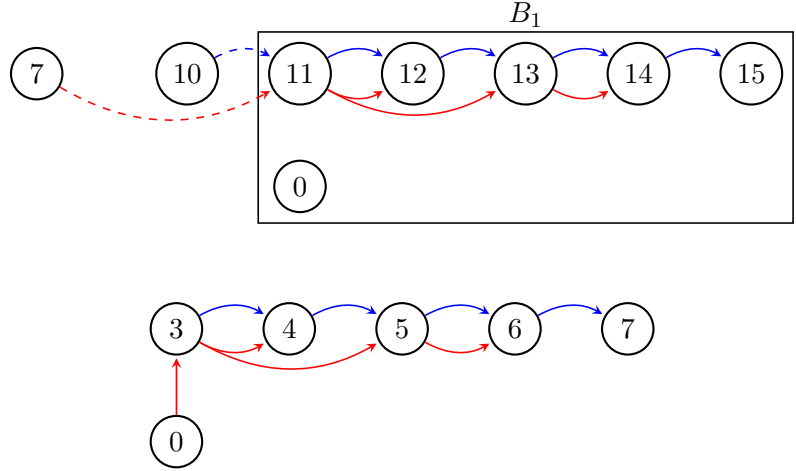
\vspace*{0.5cm}
	\centering
	\begin{tikzpicture}[> = stealth, shorten > = 1pt, semithick, node distance=1.5cm]
	\tikzstyle{crc} = [draw=black, thick, circle, fill=white]
	
	\node[crc](15){$15$};
	\node[crc](14)[left of=15]{$14$};
	\node[crc](13)[left of=14]{$13$};
	\node[crc](12)[left of=13]{$12$};
	\node[crc](11)[left of=12]{$11$};
	\node[crc](10)[left of=11]{$10$};
	\node[crc](7)[left of=10, node distance=2cm]{$7$};
	\node[crc](0I)[below of=11]{$0$};
	
	\node[draw, fit=(15) (11) (0I)](I){};
	\node[above of=I, node distance=1.5cm](I-label){$B_1$};
	
	\path[->, bend left, blue, dashed] (10) edge (11);
	\path[->, bend left, blue] (11) edge (12);
	\path[->, bend left, blue] (12) edge (13);
	\path[->, bend left, blue] (13) edge (14);
	\path[->, bend left, blue] (14) edge (15);
	
	\path[->, bend right, red, dashed] (7) edge (11);
	\path[->, bend right, red] (11) edge (12);
	\path[->, bend right, red] (11) edge (13);
	\path[->, bend right, red] (13) edge (14);
	\end{tikzpicture}
	
	\vspace*{1cm}
	\centering
	\begin{tikzpicture}[> = stealth, shorten > = 1pt, semithick, node distance=1.5cm]
	\tikzstyle{crc} = [draw=black, thick, circle, fill=white]
	
	\node[crc](7){$7$};
	\node[crc](6)[left of=7]{$6$};
	\node[crc](5)[left of=6]{$5$};
	\node[crc](4)[left of=5]{$4$};
	\node[crc](3)[left of=4]{$3$};
	\node[crc](0)[below of=3]{$0$};
	
	\path[->, bend left, blue] (3) edge (4);
	\path[->, bend left, blue] (4) edge (5);
	\path[->, bend left, blue] (5) edge (6);
	\path[->, bend left, blue] (6) edge (7);
	
	\path[->, bend right, red] (3) edge (4);
	\path[->, bend right, red] (3) edge (5);
	\path[->, bend right, red] (5) edge (6);
	\path[->, red] (0) edge (3);
	\end{tikzpicture}
	
\end{example}

\begin{proposition} \label{P1:Box II structure}
	$B_2$ of $G_{k+1}$ for $k \geq 2$ is the same as $G_{k}\backslash\rrec{I_{t}^{k}}$ with $2^{k-1}$ added to every nonzero node. 
\end{proposition}
\begin{proof}
	\begin{itemize}
		\item Nodes:
		
		It is sufficient to show that the top and bottom nodes of $G_k\backslash\rrec{I_{t}^{k}}$ with $2^{k-1}$ added match the top and bottom nodes of Box II of $G_{k+1}$ since all the nodes in between are sequential.
		\begin{align*}
		II_{b}^{k} + 2^{k-1} & = (2^{k-1}-1) + 2^{k-1}	& (I_{t}^{k} - 1) + 2^{k-1} & = (2^{k} - 2) + 2^{k-1} \\
		& = 2^{k} - 1  									& & = 2^{k+1} - (2^{k+1} - 2^{k} - 2^{k-1}) - 2 \\
		& = II_{b}^{k+1}					   				& & = 2^{k+1} - 2^{k-1} - 2 \\
		& 							   					& & = II_{t}^{k+1}
		\end{align*}
		
		\item Blue Edges:
		
		\rrec{m+2^{k-1}} and \rrec{(m+1)+2^{k-1}} are still adjacent.
		
		\item Red Edges:
		
		The $k$ digit binary representation of $II_{b}^{k} < m < I_{t}^{k}$ has a leading one by Lemma \ref{leading_one} and at least one zero since $I_{t}^{k}$ is $k$ ones. Consider two cases.
		\begin{itemize}
			\item[$\circ$] $\Pone(m)$ has a leading one in its $k$ digit binary representation:
			
			Then $\Pone(m + 2^{k-1}) = \Pone(m) + 2^{k-1}$.
			
			\item[$\circ$] $\Pone(m)$ has a leading zero in its $k$ digit binary representation:
			
			Then $m$ must have the form ``10\ldots01\ldots1" and $\Pone(m) = II_{b}^{k}$. Therefore, $m + 2^{k-1}$ has the form ``100\ldots01\ldots1" and $\Pone(m + 2^{k-1}) = II_{b}^{k+1} = II_{b}^{k} + 2^{k-1}$.
		\end{itemize}
	\end{itemize}
\end{proof}

\begin{example} \label{P1:Box II structure ex}
	$B_2$ of $G_4$ and $G_3$.
	
	\vspace*{0.5cm}
	\centering
	\begin{tikzpicture}[> = stealth, shorten > = 1pt, semithick, node distance=1.5cm]
	\tikzstyle{crc} = [draw=black, thick, circle, fill=white]
	
	\node[crc](10){$10$};
	\node[crc](9)[left of=10]{$9$};
	\node[crc](8)[left of=9]{$8$};
	\node[crc](7)[left of=8]{$7$};
	\node[crc](0)[below of=7]{$0$};
	
	\node[draw, fit=(10) (7) (0)](II){};
	\node[above of=II, node distance=1.5cm](II-label){$B_2$};
	
	\path[->, bend left, blue] (7) edge (8);
	\path[->, bend left, blue] (8) edge (9);
	\path[->, bend left, blue] (9) edge (10);
	
	\path[->, red] (0) edge (7);
	\path[->, bend right, red] (7) edge (8);
	\path[->, bend right, red] (7) edge (9);
	\path[->, bend right, red] (9) edge (10);
	\end{tikzpicture}
	
	\vspace*{1cm}
	\centering
	\begin{tikzpicture}[> = stealth, shorten > = 1pt, semithick, node distance=1.5cm]
	\tikzstyle{crc} = [draw=black, thick, circle, fill=white]
	
	\node[crc](7){$7$};
	\node[crc](6)[left of=7]{$6$};
	\node[crc](5)[left of=6]{$5$};
	\node[crc](4)[left of=5]{$4$};
	\node[crc](3)[left of=4]{$3$};
	\node[crc](0)[below of=3]{$0$};
	
	\path[->, bend left, blue] (3) edge (4);
	\path[->, bend left, blue] (4) edge (5);
	\path[->, bend left, blue] (5) edge (6);
	\path[->, bend left, blue] (6) edge (7);
	
	\path[->, bend right, red] (3) edge (4);
	\path[->, bend right, red] (3) edge (5);
	\path[->, bend right, red] (5) edge (6);
	\path[->, red] (0) edge (3);
	\end{tikzpicture}
	
\end{example}

We can make some additional observations about the structure of $G_k$.

\begin{proposition}\label{P1:One zero edge}
	The only edge from \rrec{0} goes to \rrec{II_{b}^{k}}.
\end{proposition}
\begin{proof}
	(Proof by induction)
	
	\begin{itemize}
		\item Base Cases: Observe that this is true for $G_2, G_3$, and $G_4$ in Example \ref{P1:G_2 - G_4}.
		
		\item Suppose for induction that the claim holds for $G_N$ where $N \geq 4$ and consider $G_{N+1}$. By Proposition \ref{P1:Box I structure}, $B_1$ of $G_{N+1}$ is a copy of $G_N$ except that the only edge from \rrec{0} has been changed. By Proposition \ref{P1:Box II structure}, $B_2$ of $G_{N+1}$ is a copy of $G_N\backslash\rrec{I_{t}^{k}}$, so it will have only one edge from \rrec{0} which goes to \rrec{II_{b}^{N+1}}.
	\end{itemize}
\end{proof}

\begin{proposition}\label{P1:Box II to Box I edges}
	There are only two edges from $B_2$ to $B_1$ of $G_k$, namely a blue edge from \rrec{II_{t}^{k}} to \rrec{I_{b}^{k}} and a red edge from \rrec{II_{b}^{k}} to \rrec{I_{b}^{k}}.
\end{proposition}
\begin{proof}
	Immediate from Proposition \ref{P1:Box I structure}.
\end{proof}

\paragraph*{Maximum Weight Path}

Now that the structure of the graph has been examined we will use the recursive structure to examine the maximum weight path.

\begin{theorem}
	The maximum weight path from \rrec{0} to \rrec{I^{k}_t} in $G_k$ for $k \geq 2$ is zero when $n = k$.
\end{theorem}
\begin{proof}
	(Proof by induction)
	
	\begin{itemize}
		\item Base Cases: The maximum weight path in $G_2, G_3$, and $G_4$ is zero as can be seen Example \ref{P1:G_2 - G_4} with $n = k$.
		
		\item Suppose for induction that $G_N$ has a maximum weight path of zero and consider $G_{N+1}$. By Proposition \ref{P1:One zero edge}, all paths from \rrec{0} to \rrec{I_{t}^{N+1}} must use the red edge from \rrec{0} to \rrec{II_{b}^{N+1}}. There are two cases to consider by Proposition \ref{P1:Box II to Box I edges}.
		\begin{itemize}
			\item[$\circ$] A path from \rrec{II_{b}^{N+1}} to \rrec{II_{t}^{N+1}} inside $B_2$ and then the blue edge from \rrec{II_{t}^{N+1}} to \rrec{I_{b}^{N+1}} is taken.
			
			Note that this path is a copy of a path in $G_{N}$ except that blue edges have a weight that is one less, so the maximum weight is at most zero. Furthermore, a path from \rrec{I_{b}^{N+1}} to \rrec{I_{t}^{N+1}} has a maximum weight of at most zero since it is a copy of a path in $G_N$ by Proposition \ref{P1:Box I structure} except with a ``cheaper" replacement for the edge from \rrec{0} to \rrec{II_{b}^{N}}.
			
			\item[$\circ$] The red edge from \rrec{II_{b}^{N+1}} to \rrec{I_{b}^{N+1}} is taken.
			
			Any path from \rrec{I_{b}^{N+1}} to \rrec{I_{t}^{N+1}} is a copy of a path from \rrec{II_{b}^{N}} to \rrec{I_{t}^{N}} by Proposition \ref{P1:Box I structure} except that blue edges have a weight that is one less. The blue edge from \rrec{I_{t}^{N+1}-1} to \rrec{I_{t}^{N+1}} must be taken which cancels out the extra red path. Therefore the maximum weight path is zero.
		\end{itemize}
	\end{itemize}
\end{proof}

\begin{corollary}\label{cor-rk1-prune_1_mw}
	The maximum weight path from \rrec{0} to \rrec{2^{n}-1} in $G_2 \cup G_3 \cup \cdots \cup G_n$ is zero.
\end{corollary}
\begin{proof}
	Note that if $n > k$, then the maximum weight path from \rrec{0} to \rrec{I_{t}^{k}} in $G_k$ is less than zero since all the blue edges are more negative, the red edges are unchanged, and at least one blue edge from \rrec{I_{t}^{k}-1} to \rrec{I_{t}^{k}} must be taken. Therefore, taking the edge from \rrec{0} to \rrec{II_{b}^{k}} and then a path to \rrec{I_{t}^{k}} is ``cheaper" than the edge from \rrec{0} to \rrec{II_{b}^{k+1}}. Hence the maximum weight path must take the edge from \rrec{0} to \rrec{II_{b}^{n}} and thus lies entirely within $G_n$.
\end{proof}

\begin{example}
	$G_2 \cup G_3 \cup G_4 \cup G_5$.
	
	
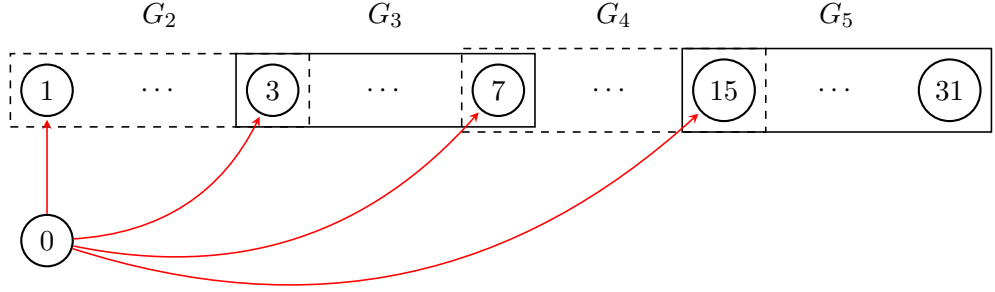
\vspace*{0.5cm}
	\centering
	\begin{tikzpicture}[> = stealth, shorten > = 1pt, semithick, node distance=1.5cm]
	\tikzstyle{rec} = [draw=black, thick, rectangle, fill=white, align=center]
	\tikzstyle{crc} = [draw=black, thick, fill=white, circle]
	
	\node[crc](31){$31$};
	\node(E1)[left of=31]{$\cdots$};
	\node[crc](15)[left of=E1]{$15$};
	\node(E2)[left of=15]{$\cdots$};
	\node[crc](7)[left of=E2]{$7$};
	\node(E3)[left of=7]{$\cdots$};
	\node[crc](3)[left of=E3]{$3$};
	\node(E4)[left of=3]{$\cdots$};
	\node[crc](1)[left of=E4]{$1$};
	\node[crc](0)[below of=1, node distance=2cm]{$0$};
	
	\node[draw, fit=(15) (31)](G5){};
	\node[above of=G5, node distance=1cm](G5-label){$G_5$};
	\node[draw, fit=(7) (15), dashed](G4){};
	\node[above of=G4, node distance=1cm](G4-label){$G_4$};
	\node[draw, fit=(3) (7)](G3){};
	\node[above of=G3, node distance=1cm](G3-label){$G_3$};
	\node[draw, fit=(1) (3), dashed](G2){};
	\node[above of=G2, node distance=1cm](G2-label){$G_2$};
	
	\path[->, red] (0) edge (1);
	\path[->, bend right, red] (0) edge (3);
	\path[->, bend right, red] (0) edge (7);
	\path[->, bend right, red] (0) edge (15);
	\end{tikzpicture}
\end{example}

By construction of the graph for $\E_{1}$, Corollary \ref{cor-rk1-prune_1_mw} gives us $\displaystyle\underset{(R,C) \in \E_1(S_1)}{\max}\; \frac{C}{R} = n+1$. Thus for any $S \in S_{1}$, the number of cases checked $C$ is at worst $(n+1)R$, so at most $n\vert S \vert$ extraneous cases are checked.

\section{Prune 2}

We will show that $\displaystyle\underset{(R,C) \in E_2(S_2)}{\max}\; \frac{C}{R} = 2$.

\begin{example}
	Let $n = 3$. Then the graph representing $\E_{1}$ is
	
	
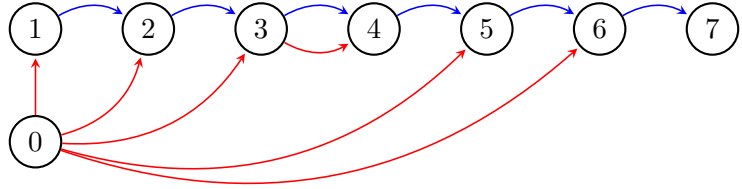
\vspace*{0.5cm}
	{\centering
		\begin{tikzpicture}[> = stealth, shorten > = 1pt, semithick, node distance=1.5cm]
		\tikzstyle{crc} = [draw=black, thick, circle, fill=white]
		
		\node[crc](7){$7$};
		\node[crc](6)[left of=7]{$6$};
		\node[crc](5)[left of=6]{$5$};
		\node[crc](4)[left of=5]{$4$};
		\node[crc](3)[left of=4]{$3$};
		\node[crc](2)[left of=3]{$2$};
		\node[crc](1)[left of=2]{$1$};
		\node[crc](0)[below of=1]{$0$};
		
		\path[->, bend left, blue] (1) edge (2);
		\path[->, bend left, blue] (2) edge (3);
		\path[->, bend left, blue] (3) edge (4);
		\path[->, bend left, blue] (4) edge (5);
		\path[->, bend left, blue] (5) edge (6);
		\path[->, bend left, blue] (6) edge (7);
		
		\path[->, bend right, red] (0) edge (2);
		\path[->, bend right, red] (0) edge (3);
		\path[->, bend right, red] (3) edge (4);
		\path[->, bend right, red] (0) edge (5);
		\path[->, bend right, red] (0) edge (6);
		\path[->, red] (0) edge (1);
		\end{tikzpicture}
		
	}
	\noindent where blue lines have weight $-1$ and red lines have weight $+1$. For $S = \{6,7 \} \in S_{2}$, $\E_{2}(S)$ would start with $j = 7$, go to $6$, go to $5$, then skip to $0$, returning $(2, 3) = (2, 2\cdot 2 - 1)$, so this particular example finishes one under par for the max ratio $C/R$ of $2$. We can convert the steps of $\E_{2}$ on $S$ into a path in the graph above by reversing the order of the $j$ values. Traveling from \rrec{0} to \rrec{5} on the red edge, then to \rrec{6} on the blue edge, then to \rrec{7} on the blue edge gives a total weight of $1 + -1 + -1 = -1$, the par number for $S$.	
\end{example}

Again, it is useful to break the graph of $\E_{2}$ into several parts. Consider the weighted directed graph $G_k$ for $k \geq 2$ with nodes \rrec{0}, \rrec{2^{k-1}-1}, \rrec{2^{k-1}}, \ldots, \rrec{2^{k}-1} with edges based on the following rules.
\begin{itemize}
	\item An edge with weight $-1$ goes from \rrec{m} to \rrec{m+1} for $2^{k-1}-1 \leq m < 2^{k}-1$.
	
	\item An edge with weight $+1$ goes from \rrec{\Ptwo(m)} to \rrec{m} for $2^{k-1} -1 \leq m < 2^{k}-1$.
\end{itemize}
Note that \rrec{0} will have no incoming edges, \rrec{2^{k-1}-1} will have one incoming edge of weight $+1$, \rrec{2^{k-1}},\ldots,\rrec{2^{k}-2} will each have two incoming edges, one of weight $-1$ and one of weight $+1$, and \rrec{2^{k}-1} will have one incoming edge of weight $-1$. There are no other edges. Also note that edges go from a smaller number to a larger. Joining $G_2$, $G_3$, \ldots, $G_n$ together and examining the maximum weight path gives the worst case performance for $\E_{2}$.

\begin{example}\label{P2:G_2 - G_4}
	$G_2$, $G_3$, and $G_4$. Blue edges have weight $-1$ and red edges have weight $+1$. The graph for $\E_{2}$ when $n=4$ is $G_2 \cup G_3 \cup G_4$.
	
	\vspace*{0.5cm}
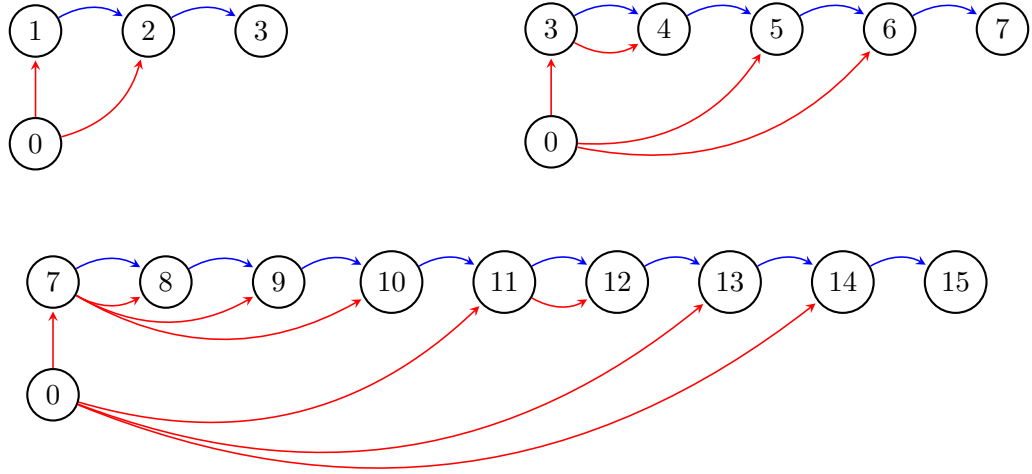
	\begin{minipage}{0.35\linewidth}
		\centering
		\begin{tikzpicture}[> = stealth, shorten > = 1pt, semithick, node distance=1.5cm]
		\tikzstyle{crc} = [draw=black, thick, circle, fill=white]
		
		\node[crc](3){$3$};
		\node[crc](2)[left of=3]{$2$};
		\node[crc](1)[left of=2]{$1$};
		\node[crc](0)[below of=1]{$0$};
		
		\path[->, bend left, blue] (1) edge (2);
		\path[->, bend left, blue] (2) edge (3);
		
		\path[->, bend right, red] (0) edge (2);
		\path[->, red] (0) edge (1);
		\end{tikzpicture}
	\end{minipage}
	\begin{minipage}{0.65\linewidth}
		\centering
		\begin{tikzpicture}[> = stealth, shorten > = 1pt, semithick, node distance=1.5cm]
		\tikzstyle{crc} = [draw=black, thick, circle, fill=white]
		
		\node[crc](7){$7$};
		\node[crc](6)[left of=7]{$6$};
		\node[crc](5)[left of=6]{$5$};
		\node[crc](4)[left of=5]{$4$};
		\node[crc](3)[left of=4]{$3$};
		\node[crc](0)[below of=3]{$0$};
		
		\path[->, bend left, blue] (3) edge (4);
		\path[->, bend left, blue] (4) edge (5);
		\path[->, bend left, blue] (5) edge (6);
		\path[->, bend left, blue] (6) edge (7);
		
		\path[->, bend right, red] (3) edge (4);
		\path[->, bend right, red] (0) edge (5);
		\path[->, bend right, red] (0) edge (6);
		\path[->, red] (0) edge (3);
		\end{tikzpicture}
	\end{minipage}
	
	\vspace*{1cm}
	\centering
	\begin{tikzpicture}[> = stealth, shorten > = 1pt, semithick, node distance=1.5cm]
	\tikzstyle{crc} = [draw=black, thick, circle, fill=white]
	
	\node[crc](15){$15$};
	\node[crc](14)[left of=15]{$14$};
	\node[crc](13)[left of=14]{$13$};
	\node[crc](12)[left of=13]{$12$};
	\node[crc](11)[left of=12]{$11$};
	\node[crc](10)[left of=11]{$10$};
	\node[crc](9)[left of=10]{$9$};
	\node[crc](8)[left of=9]{$8$};
	\node[crc](7)[left of=8]{$7$};
	\node[crc](0)[below of=7]{$0$};
	
	\path[->, bend left, blue] (7) edge (8);
	\path[->, bend left, blue] (8) edge (9);
	\path[->, bend left, blue] (9) edge (10);
	\path[->, bend left, blue] (10) edge (11);
	\path[->, bend left, blue] (11) edge (12);
	\path[->, bend left, blue] (12) edge (13);
	\path[->, bend left, blue] (13) edge (14);
	\path[->, bend left, blue] (14) edge (15);
	
	\path[->, bend right, red] (7) edge (8);
	\path[->, bend right, red] (7) edge (9);
	\path[->, bend right, red] (7) edge (10);
	\path[->, bend right, red] (0) edge (11);
	\path[->, bend right, red] (11) edge (12);
	\path[->, bend right, red] (0) edge (14);
	\path[->, bend right, red] (0) edge (13);
	\path[->, red] (0) edge (7);
	\end{tikzpicture}
\end{example}

\begin{example} \label{P2:G_5}
	$G_5$. Blue edges have weight $-1$ and red edges have weight $+1$.
	
	
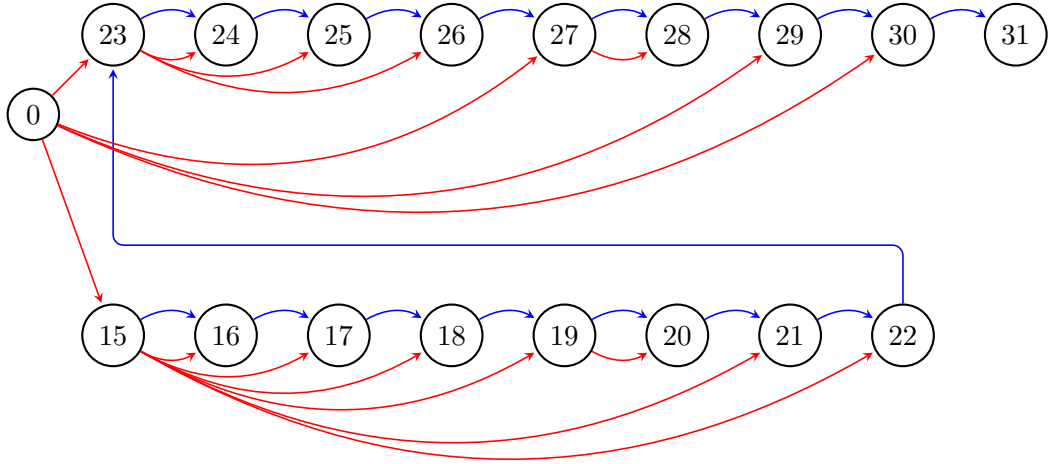
\vspace*{0.5cm}
	\centering
	\begin{tikzpicture}[> = stealth, shorten > = 1pt, semithick, node distance=1.5cm]
	\tikzstyle{crc} = [draw=black, thick, fill=white, circle]
	
	\node[crc](31){$31$};
	\foreach \x in {30,29,...,23}{
		\pgfmathtruncatemacro\y{\x+1}
		\node[crc](\x)[left of=\y]{$\x$};
	}
	
	\foreach \x in {23,...,30}{
		\pgfmathtruncatemacro\y{\x+1}
		\path[->, bend left, blue](\x) edge (\y);
	}
	
	\node[crc](22)[below of=30, node distance=4cm]{$22$};
	\foreach \x in {21,20,...,15}{
		\pgfmathtruncatemacro\y{\x+1}
		\node[crc](\x)[left of=\y]{$\x$};
	}
	\node[crc](0)[below left of=23]{$0$};
	
	\foreach \x in {15,...,21}{
		\pgfmathtruncatemacro\y{\x+1}
		\path[->, bend left, blue](\x) edge (\y);
	}
	
	\draw[->, blue, rounded corners] (22) |- ($ (27)!0.7!(19) $) -| (23);
	
	\path[->, bend right, red] (0) edge (30);
	\path[->, bend right, red] (0) edge (29);
	\path[->, bend right, red] (27) edge (28);
	\path[->, bend right, red] (0) edge (27);
	\path[->, bend right, red] (23) edge (26);
	\path[->, bend right, red] (23) edge (25);
	\path[->, bend right, red] (23) edge (24);
	\path[->, red] (0) edge (23);
	\path[->, bend right, red] (15) edge (22);
	\path[->, bend right, red] (15) edge (21);
	\path[->, bend right, red] (19) edge (20);
	\path[->, bend right, red] (15) edge (19);
	\path[->, bend right, red] (15) edge (18);
	\path[->, bend right, red] (15) edge (17);
	\path[->, bend right, red] (15) edge (16);
	\path[->, red] (0) edge (15);
	\end{tikzpicture}
	
\end{example}

\begin{definition}
	Let
	\begin{align*}
	I^k_t   & = 2^{k}-1 				  & III^k_t & = 2^{k}-2^{k-2}-2^{k-4}-2 \\
	I^k_b   & = 2^{k}-2^{k-2}-1 		  & III^k_b & = 2^{k}-2^{k-2}-2^{k-3}-1 \\
	II^k_t  & = 2^{k}-2^{k-2}-2 		  & IV^k_t  & = 2^{k}-2^{k-2}-2^{k-3}-2 \\
	II^k_b  & = 2^{k}-2^{k-2}-2^{k-4}-1 & IV^k_b  & = 2^{k-1}-1
	\end{align*}
	Then $G_k$ for $k \geq 4$ can be split into four ``boxes" as follows:
	\begin{itemize}
		\item $B_1$: The induced subgraph of $G_k$ by taking \rrec{0}, \rrec{I^k_b}, \ldots, \rrec{I^k_t}.
		
		\item $B_2$: The induced subgraph of $G_k$ by taking \rrec{0}, \rrec{II^k_b}, \ldots, \rrec{II^k_t}.
		
		\item $B_3$: The induced subgraph of $G_k$ by taking \rrec{0}, \rrec{III^k_b}, \ldots, \rrec{III^k_t}.
		
		\item $B_4$: The induced subgraph of $G_k$ by taking \rrec{0}, \rrec{IV^k_b}, \ldots, \rrec{IV^k_t}.
	\end{itemize}
\end{definition}

\begin{example} \label{P2:G_4 boxes ex}
	Boxes of $G_4$. Blue edges have weight $-1$ and red edges have weight $+1$. Dashed edges go between boxes.
	
	
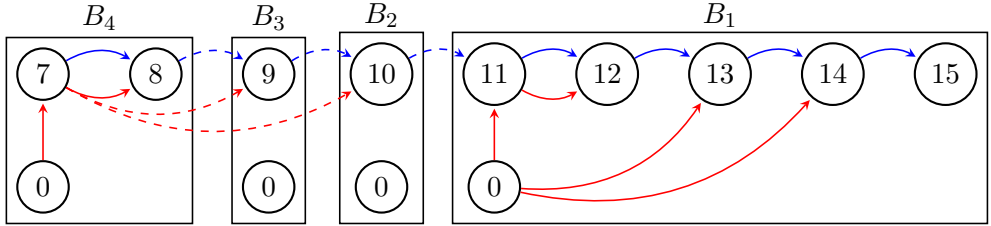
\vspace*{0.5cm}
	\centering
	\begin{tikzpicture}[> = stealth, shorten > = 1pt, semithick, node distance=1.5cm]
	\tikzstyle{crc} = [draw=black, thick, circle, fill=white]
	
	\node[crc](15){$15$};
	\node[crc](14)[left of=15]{$14$};
	\node[crc](13)[left of=14]{$13$};
	\node[crc](12)[left of=13]{$12$};
	\node[crc](11)[left of=12]{$11$};
	\node[crc](10)[left of=11]{$10$};
	\node[crc](9)[left of=10]{$9$};
	\node[crc](8)[left of=9]{$8$};
	\node[crc](7)[left of=8]{$7$};
	\node[crc](0I)[below of=11]{$0$};
	\node[crc](0II)[below of=10]{$0$};
	\node[crc](0III)[below of=9]{$0$};
	\node[crc](0IV)[below of=7]{$0$};
	
	\node[draw, fit=(15) (11) (0I)](I){};
	\node[above of=I, node distance=1.5cm](I-label){$B_1$};
	\node[draw, fit=(10) (0II)](II){};
	\node[above of=II, node distance=1.5cm](II-label){$B_2$};
	\node[draw, fit=(9) (0III)](III){};
	\node[above of=III, node distance=1.5cm](III-label){$B_3$};
	\node[draw, fit=(7) (8) (0IV)](IV){};
	\node[above of=IV, node distance=1.5cm](IV-label){$B_4$};
	
	\path[->, bend left, blue] (7) edge (8);
	\path[->, bend left, blue, dashed] (8) edge (9);
	\path[->, bend left, blue, dashed] (9) edge (10);
	\path[->, bend left, blue, dashed] (10) edge (11);
	\path[->, bend left, blue] (11) edge (12);
	\path[->, bend left, blue] (12) edge (13);
	\path[->, bend left, blue] (13) edge (14);
	\path[->, bend left, blue] (14) edge (15);
	
	\path[->, bend right, red] (7) edge (8);
	\path[->, bend right, red, dashed] (7) edge (9);
	\path[->, bend right, red, dashed] (7) edge (10);
	\path[->, red] (0I) edge (11);
	\path[->, bend right, red] (11) edge (12);
	\path[->, bend right, red] (0I) edge (14);
	\path[->, bend right, red] (0I) edge (13);
	\path[->, red] (0IV) edge (7);
	\end{tikzpicture}
\end{example}

\begin{example}\label{P2:G_5 boxes ex}
	Boxes of $G_5$. Blue edges have weight $-1$ and red edges have weight $+1$. Dashed edges go between boxes.
	
	
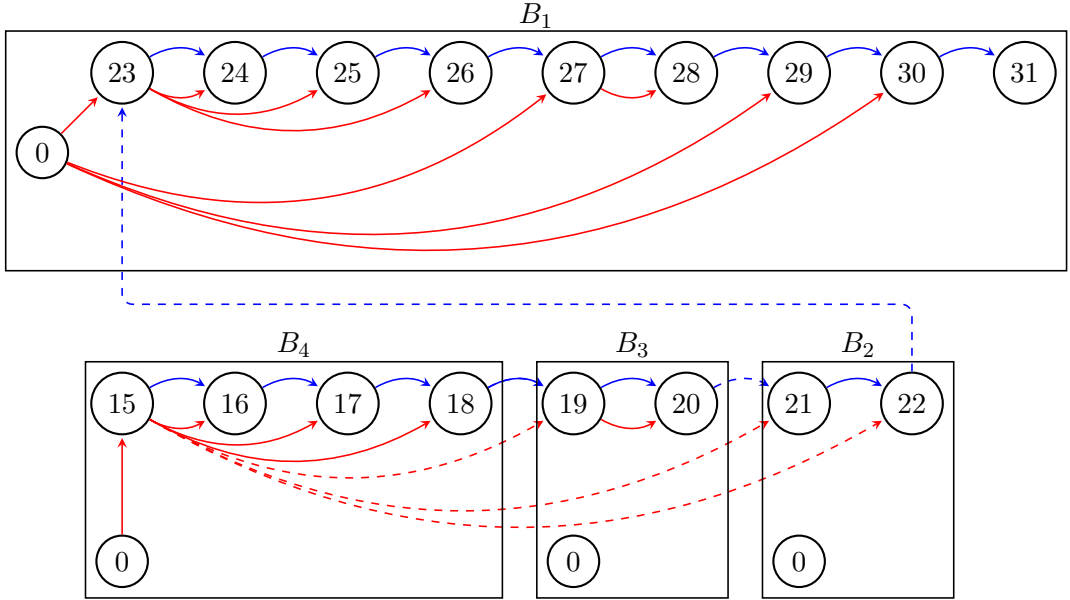
\vspace*{0.5cm}
	\centering
	\begin{tikzpicture}[> = stealth, shorten > = 1pt, semithick, node distance=1.5cm]
	\tikzstyle{crc} = [draw=black, thick, fill=white, circle]
	
	\node[crc](31){$31$};
	\foreach \x in {30,29,...,23}{
		\pgfmathtruncatemacro\y{\x+1}
		\node[crc](\x)[left of=\y]{$\x$};
	}
	
	\node[crc](22)[below of=30, node distance=4.4cm]{$22$};
	\foreach \x in {21,20,...,15}{
		\pgfmathtruncatemacro\y{\x+1}
		\node[crc](\x)[left of=\y]{$\x$};
	}
	
	\node[crc](0I)[below left of=23]{$0$};
	\node(Ispacer)[below of=0I, node distance=1.3cm]{};
	\node[crc](0II)[below of=21, node distance=2.1cm]{$0$};
	\node[crc](0III)[below of=19, node distance=2.1cm]{$0$};
	\node[crc](0IV)[below of=15, node distance=2.1cm]{$0$};
	
	\foreach \x in {23,...,30}{
		\pgfmathtruncatemacro\y{\x+1}
		\path[->, bend left, blue](\x) edge (\y);
	}
	
	\foreach \x in {15,...,18}{
		\pgfmathtruncatemacro\y{\x+1}
		\path[->, bend left, blue](\x) edge (\y);
	}
	
	\path[->, bend left, blue, dashed](18) edge (19);
	\path[->, bend left, blue](19) edge (20);
	\path[->, bend left, blue, dashed](20) edge (21);
	\path[->, bend left, blue](21) edge (22);
	\draw[->, blue, rounded corners, dashed] (22) |- ($ (27)!0.7!(19) $) -| (23);
	
	\path[->, bend right, red] (0I) edge (30);
	\path[->, bend right, red] (0I) edge (29);
	\path[->, bend right, red] (27) edge (28);
	\path[->, bend right, red] (0I) edge (27);
	\path[->, bend right, red] (23) edge (26);
	\path[->, bend right, red] (23) edge (25);
	\path[->, bend right, red] (23) edge (24);
	\path[->, red] (0I) edge (23);
	\path[->, bend right, red, dashed] (15) edge (22);
	\path[->, bend right, red, dashed] (15) edge (21);
	\path[->, bend right, red] (19) edge (20);
	\path[->, bend right, red, dashed] (15) edge (19);
	\path[->, bend right, red] (15) edge (18);
	\path[->, bend right, red] (15) edge (17);
	\path[->, bend right, red] (15) edge (16);
	\path[->, red] (0IV) edge (15);
	
	\node[draw, fit=(31) (0I) (Ispacer)](I){};
	\node[above of=I, node distance=1.8cm](I-label){$B_1$};
	\node[draw, fit=(22) (0II)](II){};
	\node[above of=II, node distance=1.8cm](II-label){$B_2$};
	\node[draw, fit=(20) (0III)](III){};
	\node[above of=III, node distance=1.8cm](III-label){$B_3$};
	\node[draw, fit=(18) (0IV)](IV){};
	\node[above of=IV, node distance=1.8cm](IV-label){$B_4$};
	\end{tikzpicture}
\end{example}

\paragraph*{Recursive Structure}
$G_{k+1}$ can be constructed from $G_{k}$ using the following propositions.

\begin{proposition} \label{P2:Box I structure}
	$B_1$ of $G_{k+1}$ for $k \geq 3$ is the same as $G_{k}$ with $2^{k}$ added to every nonzero node except that \rrec{I_{b}^{k+1}} has a blue edge from \rrec{II_{t}^{k+1}}. 
\end{proposition}
\begin{proof}
	\begin{itemize}
		\item Nodes:
		
		It is sufficient to show that the top and bottom nodes of $G_k$ with $2^{k}$ added match the top and bottom nodes of $B_1$ of $G_{k+1}$ since all the nodes in between are sequential.
		\begin{align*}
		IV_{b}^{k} + 2^{k} & = (2^{k-1}-1) + 2^{k}     & I_{t}^{k} + 2^{k} & = (2^{k} - 1) + 2^{k} \\
		& = 2^{k+1} - (2^{k+1} - 2^{k} - 2^{k-1}) - 1  & & = 2^{k+1} - 1 \\
		& = 2^{k+1} - 2^{k-1} - 1 					   & & = I_{t}^{k+1} \\
		& = I_{b}^{k+1}								   & &
		\end{align*}
		
		\item Blue Edges:
		
		\rrec{m+2^{k}} and \rrec{(m+1)+2^{k}} are still adjacent. Also \rrec{I_{b}^{k+1}} and \rrec{II_{t}^{k+1}} are adjacent.
		
		\item Red Edges:
		
		Note that adding $2^{k}$ to $0 < m < 2^{k}$ is the same as left-appending a one to the $k$ digit binary representation of $m$. Consider two subcases.
		\begin{itemize}
			\item[$\circ$] $m$ has $< 2$ zeros in its $k$ digit binary representation where $2^{k-1}-1 \leq m < 2^{k} - 1$:
			
			Then $\Ptwo(m) = \Ptwo(m+2^{k}) = 0$.
			
			\item[$\circ$] $m$ has $\geq 2$ zeros in its $k$ digit binary representation where $2^{k-1}-1 \leq m < 2^{k} - 1$:
			
			By Lemma \ref{leading_one}, $m+2^{k}$ has a binary representation that starts with ``11." Thus $\Ptwo(m+2^{k}) = 2^{k}+\Ptwo(m)$ since the left-most one is unaffected by $\Ptwo$.
		\end{itemize}
	\end{itemize}
\end{proof}

\begin{example} \label{P2:Box I structure ex}
	$B_1$ of $G_4$ and $G_3$.
	
	
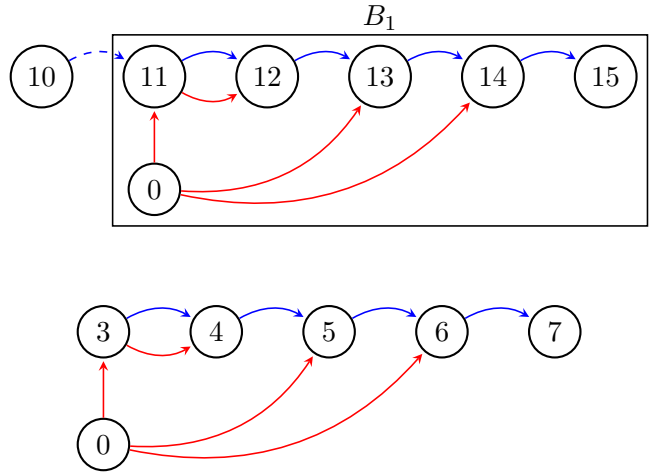
\vspace*{0.5cm}
	\centering
	\begin{tikzpicture}[> = stealth, shorten > = 1pt, semithick, node distance=1.5cm]
	\tikzstyle{crc} = [draw=black, thick, circle, fill=white]
	
	\node[crc](15){$15$};
	\node[crc](14)[left of=15]{$14$};
	\node[crc](13)[left of=14]{$13$};
	\node[crc](12)[left of=13]{$12$};
	\node[crc](11)[left of=12]{$11$};
	\node[crc](10)[left of=11]{$10$};
	\node[crc](0I)[below of=11]{$0$};
	
	\node[draw, fit=(15) (11) (0I)](I){};
	\node[above of=I, node distance=1.5cm](I-label){$B_1$};
	
	\path[->, bend left, blue, dashed] (10) edge (11);
	\path[->, bend left, blue] (11) edge (12);
	\path[->, bend left, blue] (12) edge (13);
	\path[->, bend left, blue] (13) edge (14);
	\path[->, bend left, blue] (14) edge (15);
	
	\path[->, red] (0I) edge (11);
	\path[->, bend right, red] (11) edge (12);
	\path[->, bend right, red] (0I) edge (14);
	\path[->, bend right, red] (0I) edge (13);
	\end{tikzpicture}
	
	\vspace*{1cm}
	\centering
	\begin{tikzpicture}[> = stealth, shorten > = 1pt, semithick, node distance=1.5cm]
	\tikzstyle{crc} = [draw=black, thick, circle, fill=white]
	
	\node[crc](7){$7$};
	\node[crc](6)[left of=7]{$6$};
	\node[crc](5)[left of=6]{$5$};
	\node[crc](4)[left of=5]{$4$};
	\node[crc](3)[left of=4]{$3$};
	\node[crc](0)[below of=3]{$0$};
	
	\path[->, bend left, blue] (3) edge (4);
	\path[->, bend left, blue] (4) edge (5);
	\path[->, bend left, blue] (5) edge (6);
	\path[->, bend left, blue] (6) edge (7);
	
	\path[->, bend right, red] (3) edge (4);
	\path[->, bend right, red] (0) edge (5);
	\path[->, bend right, red] (0) edge (6);
	\path[->, red] (0) edge (3);
	\end{tikzpicture}
	
\end{example}

\begin{proposition} \label{P2:Box II structure}
	$B_2 \cup$ \rrec{IV^{k+1}_b} of $G_{k+1}$ for $k \geq 4$ is the same as $B_2 \cup B_3 \cup$ \rrec{IV^k_b} of $G_k$ with $3\cdot 2^{k-2}$ added to every node $III_{b}^{k} \leq m \leq II_{t}^{k}$ and adding $2^{k-1}$ to $IV^{k}_b$.
\end{proposition}
\begin{proof}
	\begin{itemize}
		\item Nodes:
		
		It is sufficient to show that the top and bottom nodes of $B_2 \cup B_3$ of $G_k$ with $3\cdot 2^{k-2}$ added match the top and bottom nodes of $B_2$ of $G_{k+1}$, since all the nodes in between are sequential, and that $IV_{b}^k + 2^{k-1} = IV_{b}^{k+1}$.
		\begin{align*}
		III_{b}^{k} + 3\cdot 2^{k-2} & = (2^{k} - 2^{k-2} - 2^{k-3} - 1) +  3\cdot 2^{k-2} \\
		& = (2^{k} - 2^{k-2} - 2^{k-3} - 1) + (2^{k-1} + 2^{k-2}) \\
		& = 2^{k+1} - (2^{k+1} - 2^{k} - 2^{k-1} ) - 2^{k-3} - 1 \\
		& = 2^{k+1} - 2^{k-1} - 2^{k-3} - 1 \\
		& = II_{b}^{k+1}
		\end{align*}
		
		\begin{align*}
		II_{t}^{k} +  3\cdot 2^{k-2} & = (2^{k} - 2^{k-2} - 2) +  3\cdot 2^{k-2} \\
		& = (2^{k} - 2^{k-2} - 2) + (2^{k-1} + 2^{k-2}) \\
		& = 2^{k+1} - (2^{k+1} - 2^{k} - 2^{k-1}) - 2 \\
		& = 2^{k+1} - 2^{k-1} - 2 \\
		& = II_{t}^{k+1}
		\end{align*}
		
		\begin{align*}
		IV_{b}^{k} + 2^{k-1} & = (2^{k-1} - 1) + 2^{k-1} \\
		& = 2^{k} - 1 \\
		& = IV_{b}^{k+1}
		\end{align*}
		
		\item Blue Edges:
		
		\rrec{m + 3\cdot 2^{k-2}} and \rrec{(m+1) + 3\cdot 2^{k-2}} are still adjacent.
		
		\item Red Edges:
		
		By Lemma \ref{leading_101}, every node in $B_2 \cup B_3$ of $G_k$ except \rrec{III_{b}^{k}} starts with ``101" and has at least one more zero. Adding $3\cdot 2^{k-2} = 2^{k-1} + 2^{k}$ means that these nodes now have $k+1$ digits and start with ``1011." Consider two subcases.
		\begin{itemize}
			\item[$\circ$] $m$ has $2$ zeros in its $k$ digit binary representation where 
			$III_{b}^{k} < m \leq II_{t}^{k}$:
			
			Then $\Ptwo(m) = IV_{b}^{k}$ and $\Ptwo(m + 3\cdot 2^{k-2}) = IV_{b}^{k+1}$.
			
			\item[$\circ$] $m$ has $>2$ zeros in its $k$ digit binary representation where $III_{b}^{k} < m \leq II_{t}^{k}$:
			
			Then $\Ptwo(m + 3\cdot2^{k-2}) = \Ptwo(m) + 3\cdot 2^{k-2}$ since the leading ``10" of $m$ is unaffected by $\Ptwo$.
		\end{itemize}
		
		By Lemma \ref{leading_100}, $III_{b}^{k} + 3\cdot 2^{k-2} = III_{b}^{k} + 2^{k-1} + 2^{k-2}$ starts with ``1010" followed by $k-3$ ones. Thus $\Ptwo(III_{b}^{k}) = IV_{b}^{k}$ and $\Ptwo(III_{b}^{k} + 3\cdot 2^{k-2}) = IV_{b}^{k+1}$.
	\end{itemize}
\end{proof}

\begin{example} \label{P2:Box II structure ex}
	$B_2 \cup$ \rrec{31} of $G_6$ and $B_2 \cup B_3 \cup$ \rrec{15} of $G_5$.
	
	
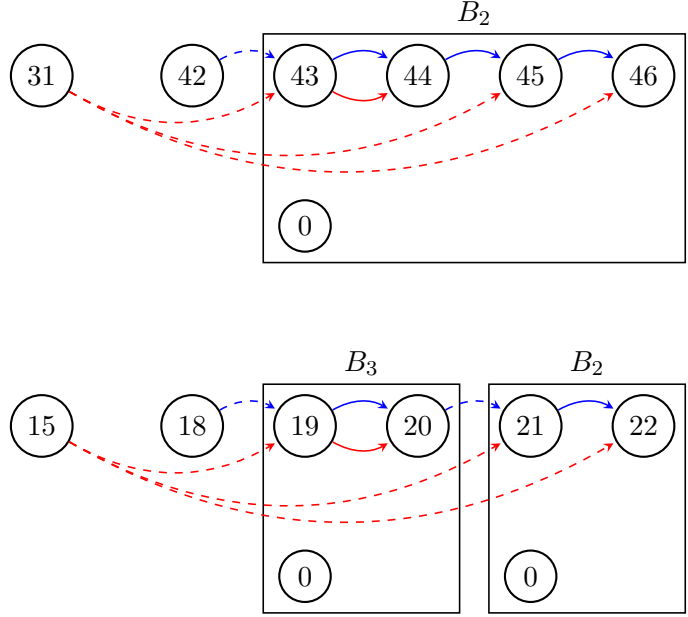
\vspace*{0.5cm}
	\centering
	\begin{tikzpicture}[> = stealth, shorten > = 1pt, semithick, node distance=1.5cm]
	\tikzstyle{crc} = [draw=black, thick, fill=white, circle]
	
	\node[crc](46){$46$};
	\node[crc](45)[left of=46]{$45$};
	\node[crc](44)[left of=45]{$44$};
	\node[crc](43)[left of=44]{$43$};
	\node[crc](42)[left of=43]{$42$};
	\node[crc](31)[left of=42, node distance=2cm]{$31$};
	\node[crc](0)[below of=43, node distance=2cm]{$0$};
	
	\node[draw, fit=(46) (43) (0)](II){};
	\node[above of=II, node distance=1.8cm](II-label){$B_2$};
	
	\path[->, bend left, blue, dashed] (42) edge (43);
	\path[->, bend left, blue] (43) edge (44);
	\path[->, bend left, blue] (44) edge (45);
	\path[->, bend left, blue] (45) edge (46);
	
	\path[->, bend right, red] (43) edge (44);
	\path[->, bend right, red, dashed] (31) edge (43);
	\path[->, bend right, red, dashed] (31) edge (45);
	\path[->, bend right, red, dashed] (31) edge (46);
	\end{tikzpicture}
	
	\vspace*{1cm}
	\centering
	\begin{tikzpicture}[> = stealth, shorten > = 1pt, semithick, node distance=1.5cm]
	\tikzstyle{crc} = [draw=black, thick, circle, fill=white]
	
	\node[crc](22){$22$};
	\node[crc](21)[left of=22]{$21$};
	\node[crc](20)[left of=21]{$20$};
	\node[crc](19)[left of=20]{$19$};
	\node[crc](18)[left of=19]{$18$};
	\node[crc](15)[left of=18, node distance=2cm]{$15$};
	\node[crc](0II)[below of=21, node distance=2cm]{$0$};
	\node[crc](0III)[below of=19, node distance=2cm]{$0$};
	
	\node[draw, fit=(22) (21) (0II)](II){};
	\node[above of=II, node distance=1.8cm](II-label){$B_2$};
	\node[draw, fit=(20) (19) (0III)](III){};
	\node[above of=III, node distance=1.8cm](III-label){$B_3$};
	
	\path[->, bend left, blue] (21) edge (22);
	\path[->, bend left, blue, dashed] (20) edge (21);
	\path[->, bend left, blue] (19) edge (20);
	\path[->, bend left, blue, dashed] (18) edge (19);
	
	\path[->, bend right, red] (19) edge (20);
	\path[->, bend right, red, dashed] (15) edge (19);
	\path[->, bend right, red, dashed] (15) edge (21);
	\path[->, bend right, red, dashed] (15) edge (22);
	\end{tikzpicture}
	
\end{example}

\begin{proposition} \label{P2:Box III structure}
	$B_3$ of $G_{k+1}$ for $k \geq 4$ is the same as $B_4$ of $G_k$ with $3\cdot 2^{k-2}$ added to every node $IV_{b}^{k} \leq m \leq IV_{t}^{k}$ except \rrec{III^{k+1}_b} has a red edge from \rrec{IV^{k+1}_b} instead of \rrec{0} and a new blue edge from \rrec{IV^{k+1}_t}.
\end{proposition}
\begin{proof}
	\begin{itemize}
		\item Nodes:
		
		It is sufficient to show that the top and bottom nodes of $B_4$ of $G_k$ with $3\cdot 2^{k-2}$ added match the top and bottom nodes of $B_3$ of $G_{k+1}$ since all the nodes in between are sequential.
		\begin{align*}
		IV_{b}^{k} + 3\cdot 2^{k-2} & = (2^{k-1} - 1) + 3\cdot 2^{k-2} \\
		& = (2^{k-1} - 1) + (2^{k-1} + 2^{k-2}) \\
		& = 2^{k} + 2^{k-2} - 1 \\
		& = 2^{k+1} - (2^{k+1} - 2^{k}) + 2^{k-2} - 1 \\
		& = 2^{k+1} - 2^{k} + 2^{k-2} - 1 \\
		& = 2^{k+1} - (2^{k-1} + 2^{k-1}) + 2^{k-2} - 1 \\
		& = 2^{k+1} - 2^{k-1} - (2^{k-1} - 2^{k-2}) - 1 \\
		& = 2^{k+1} - 2^{k-1} - 2^{k-2} - 1 \\
		& = III_{b}^{k+1}
		\end{align*}
		
		\begin{align*}
		IV_{t}^{k} + 3\cdot 2^{k-2} & = (2^{k} - 2^{k-2} - 2^{k-3} - 2) + 3\cdot 2^{k-2} \\
		& = (2^{k} - 2^{k-2} - 2^{k-3} - 2) + (2^{k-1} + 2^{k-2}) \\
		& = 2^{k} + 2^{k-1} - 2^{k-3} - 2 \\
		& = 2^{k+1} - (2^{k+1} - 2^{k} - 2^{k-1}) - 2^{k-3} - 2 \\
		& = 2^{k+1} - 2^{k-1} - 2^{k-3} - 2 \\
		& = III_{t}^{k+1}
		\end{align*}
		
		\item Blue Edges:
		
		\rrec{m + 3\cdot 2^{k-2}} and \rrec{(m+1) + 3\cdot 2^{k-2}} are still adjacent. Also, \rrec{IV_{t}^{k+1}} and \rrec{III_{b}^{k+1}} are adjacent.
		
		\item Red Edges:
		
		By Lemma \ref{leading_100}, the $k$ digit binary representation for $2^{k-1} \leq m \leq IV_{t}^{k}$ starts with ``100." Thus $m + 3\cdot 2^{k-2}$ has $k+1$ digits and starts with ``1010." Consider two cases.
		\begin{itemize}
			\item[$\circ$] The $k$ digit binary representation of $\Ptwo(m)$ for $2^{k-1} \leq m \leq IV_{t}^{k}$ starts with ``100":
			
			Then $\Ptwo(m + 3\cdot 2^{k-2}) = \Ptwo(m) + 3\cdot 2^{k-2}$.
			
			\item[$\circ$] The $k$ digit binary representation of $\Ptwo(m)$ for $2^{k-1} \leq m \leq IV_{t}^{k}$ is  ``011\ldots1" = $IV_{b}^{k}$:
			
			Then $\Ptwo(m + 3\cdot 2^{k-2}) = ``10011\ldots1" = III_{b}^{k+1}$.
		\end{itemize}
		
		Also, $\Ptwo(III_{b}^{k+1}) = ``011\ldots1 = IV_{b}^{k+1}$.
	\end{itemize}
\end{proof}

\begin{example} \label{P2:Box III structure ex}
	$B_3$ of $G_6$ and $B_4$ of $G_5$.
	
	
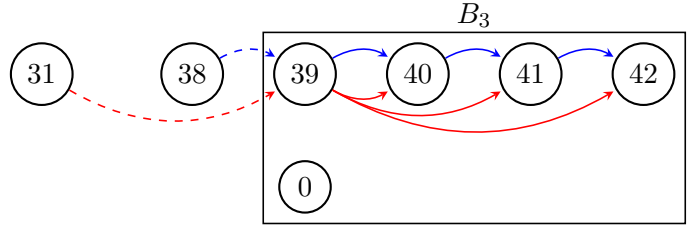
\vspace*{0.5cm}
	\centering
	\begin{tikzpicture}[> = stealth, shorten > = 1pt, semithick, node distance=1.5cm]
	\tikzstyle{crc} = [draw=black, thick, fill=white, circle]
	
	\node[crc](42){$42$};
	\node[crc](41)[left of=42]{$41$};
	\node[crc](40)[left of=41]{$40$};
	\node[crc](39)[left of=40]{$39$};
	\node[crc](38)[left of=39]{$38$};
	\node[crc](31)[left of=38, node distance=2cm]{$31$};
	\node[crc](0)[below of=39]{$0$};
	
	\node[draw, fit=(42) (39) (0)](III){};
	\node[above of=III, node distance=1.5cm](III-label){$B_3$};
	
	\path[->, bend left, blue] (39) edge (40);
	\path[->, bend left, blue] (40) edge (41);
	\path[->, bend left, blue] (41) edge (42);
	\path[->, bend left, blue, dashed] (38) edge (39);
	
	\path[->, bend right, red] (39) edge (40);
	\path[->, bend right, red] (39) edge (41);
	\path[->, bend right, red] (39) edge (42);
	\path[->, bend right, red, dashed] (31) edge (39);
	\end{tikzpicture}
	
	
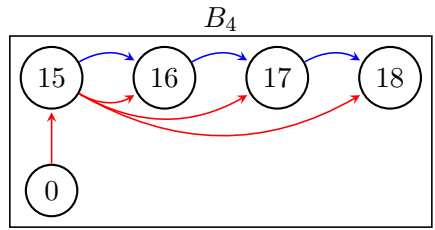
\vspace*{1cm}
	\centering
	\begin{tikzpicture}[> = stealth, shorten > = 1pt, semithick, node distance=1.5cm]
	\tikzstyle{crc} = [draw=black, thick, circle, fill=white]
	
	\node[crc](18){$18$};
	\node[crc](17)[left of=18]{$17$};
	\node[crc](16)[left of=17]{$16$};
	\node[crc](15)[left of=16]{$15$};
	\node[crc](0)[below of=15]{$0$};
	
	\node[draw, fit=(18) (15) (0)](IV){};
	\node[above of=IV, node distance=1.5cm](IV-label){$B_4$};
	
	\path[->, bend left, blue] (15) edge (16);
	\path[->, bend left, blue] (16) edge (17);
	\path[->, bend left, blue] (17) edge (18);
	
	\path[->, bend right, red] (15) edge (16);
	\path[->, bend right, red] (15) edge (17);
	\path[->, bend right, red] (15) edge (18);
	\path[->, red] (0) edge (15);
	\end{tikzpicture}
\end{example}

\begin{proposition} \label{P2:Box IV structure}
	$B_4$ of $G_{k+1}$ for $k \geq 4$ is the same as $B_2 \cup B_3 \cup B_4$ of $G_{k}$ with $2^{k-1}$ added to every nonzero node.
\end{proposition}
\begin{proof}
	\begin{itemize}
		\item Nodes:
		
		It is sufficient to show that the top and bottom nodes of $B_2 \cup B_3 \cup B_4$ of $G_k$ with $2^{k-1}$ added match the top and bottom nodes of $B_4$ of $G_{k+1}$ since all the nodes in between are sequential.
		\begin{align*}
		IV_{b}^{k} + 2^{k} & = (2^{k-1} - 1) + 2^{k-1} \\
		& =  2^{k} - 1 \\
		& = IV_{b}^{k+1}
		\end{align*}
		
		\begin{align*}
		II_{t}^{k} + 2^{k-1} & = (2^{k} - 2^{k-2} - 2) + 2^{k-1} \\
		& = 2^{k+1} - (2^{k+1} - 2^{k} - 2^{k-1}) - 2^{k-2} - 2 \\
		& = 2^{k+2} - 2^{k-1} - 2^{k-2} - 2 \\
		& = IV_{t}^{k+1}
		\end{align*}
		
		\item Blue Edges:
		
		\rrec{m + 2^{k-1}} and \rrec{(m+1) + 2^{k-1}} are still adjacent.
		
		\item Red Edges:
		
		By Lemma \ref{leading_10}, every node in $B_2 \cup B_3 \cup B_4$ of $G_k$ except \rrec{IV_{b}^{k}} has $k$ digits and starts with ``10." Adding $2^{k-1}$ means they will have $k+1$ digits and start with ``100." Consider two cases.
		\begin{itemize}
			\item[$\circ$] The $k$ digit representation of $\Ptwo(m)$ for $IV_{b}^{k} < m \leq II_{t}^{k}$ starts with ``10":
			
			Then $\Ptwo(m + 2^{k-1}) = \Ptwo(m) + 2^{k-1}$.
			
			\item[$\circ$] The $k$ digit representation of $\Ptwo(m)$ for $IV_{b}^{k} < m \leq II_{t}^{k}$ is $``011\ldots1" = IV_{b}^{k}$:
			
			Then $\Ptwo(m + 2^{k-1}) = IV_{b}^{k+1}$.
		\end{itemize}
		
		Also, $\Ptwo(IV_{b}^{k+1}) = 0$.
	\end{itemize}
\end{proof}

\begin{example} \label{P2:Box IV structure ex}
	$B_4$ of $G_6$ and $B_2 \cup B_3 \cup B_4$ of $G_{5}$.
	
	
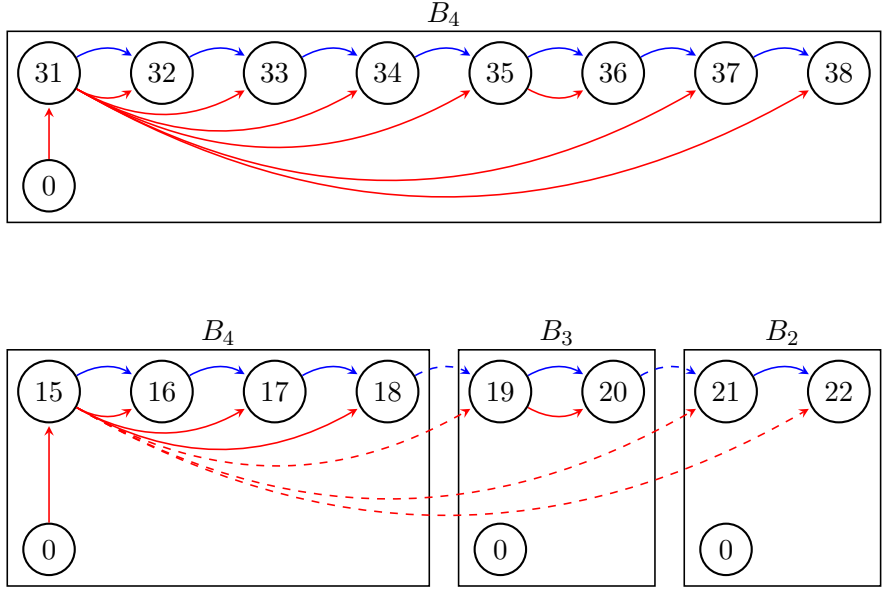
\vspace*{0.5cm}
	\centering
	\begin{tikzpicture}[> = stealth, shorten > = 1pt, semithick, node distance=1.5cm]
	\tikzstyle{crc} = [draw=black, thick, fill=white, circle]
	
	\node[crc](38){$38$};
	\foreach \x in {37,36,...,31}{
		\pgfmathtruncatemacro\y{\x+1}
		\node[crc](\x)[left of=\y]{$\x$};
	}
	\node[crc](0)[below of=31]{$0$};
	
	\foreach \x in {31,...,37}{
		\pgfmathtruncatemacro\y{\x+1}
		\path[->, bend left, blue](\x) edge (\y);
	}
	
	\node[draw, fit=(38) (31) (0)](IV){};
	\node[above of=IV, node distance=1.5cm](IV-label){$B_4$};
	
	\foreach \x in {32,...,35}
	\path[->, bend right, red](31) edge (\x);
	
	\path[->, bend right, red](35) edge (36);
	\path[->, bend right, red](31) edge (37);
	\path[->, bend right, red](31) edge (38);
	\path[->, red](0) edge (31);
	\end{tikzpicture}
	
	\vspace*{1cm}
	\centering
	\begin{tikzpicture}[> = stealth, shorten > = 1pt, semithick, node distance=1.5cm]
	\tikzstyle{crc} = [draw=black, thick, fill=white, circle]
	
	\node[crc](22){$22$};
	\foreach \x in {21,20,...,15}{
		\pgfmathtruncatemacro\y{\x+1}
		\node[crc](\x)[left of=\y]{$\x$};
	}
	\node[crc](0II)[below of=21, node distance=2.1cm]{$0$};
	\node[crc](0III)[below of=19, node distance=2.1cm]{$0$};
	\node[crc](0IV)[below of=15, node distance=2.1cm]{$0$};
	
	\node[draw, fit=(22) (21) (0II)](II){};
	\node[above of=II, node distance=1.8cm](II-label){$B_2$};
	\node[draw, fit=(20) (19) (0III)](III){};
	\node[above of=III, node distance=1.8cm](III-label){$B_3$};
	\node[draw, fit=(18) (15) (0IV)](IV){};
	\node[above of=IV, node distance=1.8cm](IV-label){$B_4$};
	
	\path[->, bend left, blue](15) edge (16);
	\path[->, bend left, blue](16) edge (17);
	\path[->, bend left, blue](17) edge (18);
	\path[->, bend left, blue, dashed](18) edge (19);
	\path[->, bend left, blue](19) edge (20);
	\path[->, bend left, blue, dashed](20) edge (21);
	\path[->, bend left, blue](21) edge (22);
	
	\path[->, bend right, red](15) edge (16);
	\path[->, bend right, red](15) edge (17);
	\path[->, bend right, red](15) edge (18);
	\path[->, bend right, red, dashed](15) edge (19);
	\path[->, bend right, red](19) edge (20);
	\path[->, bend right, red, dashed](15) edge (21);
	\path[->, bend right, red, dashed](15) edge (22);
	\path[->, red](0IV) edge (15);
	\end{tikzpicture}
\end{example}

\vspace*{1cm}

We can make some additional observations about the structure of $G_k$.

\begin{proposition} \label{P2:Box II-IV zero edge}
	$B_2 \cup B_3$ of $G_k$ has no edge from \rrec{0} and $B_4$ of $G_k$ has only one edge from \rrec{0} which goes to \rrec{IV_{b}^{k}} for $k \geq 4$.
\end{proposition}
\begin{proof}
	(Proof by induction)
	\begin{itemize}
		\item Base case: Observe that this is true for $G_4$ and $G_5$ in examples \ref{P2:G_4 boxes ex} and \ref{P2:Box IV structure ex}.
		
		\item Suppose the proposition is true for $k \geq 4$. Consider $G_{k+1}.$
		
		By Proposition \ref{P2:Box II structure}, $B_2$ of $G_{k+1}$ is a copy of $B_2 \cup B_3$ of $G_k$, so will have no edges from \rrec{0}.
		
		By Proposition \ref{P2:Box III structure}, $B_3$ of $G_{k+1}$ is a copy of $B_4$ of $G_k$ except that the (only) edge from \rrec{0} has been changed, so it will have no edges from \rrec{0}.
		
		By Proposition \ref{P2:Box IV structure}, $B_4$ of $G_{k+1}$ is a copy of $B_2 \cup B_3 \cup B_4$ of $G_k$, so it will have only one edge from \rrec{0} which goes to \rrec{IV_{b}^{k+1}}.
	\end{itemize}
\end{proof}

\begin{proposition} \label{P2:Box I outside edge}
	$B_1$ of $G_k$ has only one edge from outside of $B_1$ which is the edge with weight $-1$ from \rrec{II_{t}^{k}} to \rrec{I_{b}^{k}}.
\end{proposition}
\begin{proof}
	Immediate from Proposition \ref{P2:Box I structure}.
\end{proof}

\begin{example}
	Putting all of the above Propositions together allows for $G_{k+1}$ to be built from $G_{k}$
	
	
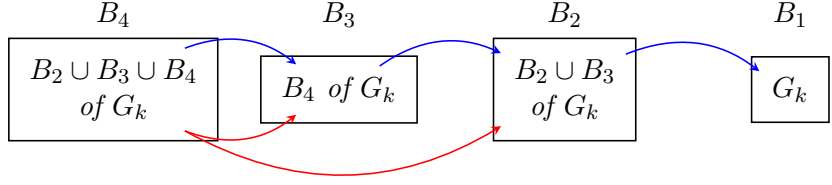
\vspace*{0.5cm}
	\centering
	\begin{tikzpicture}[> = stealth, shorten > = 1pt, semithick, node distance=3cm]
	
	\node[align=center](BI){$G_{k}$};
	\node[align=center](BII)[left of=BI]{$B_2 \cup B_3$ \\ of $G_{k}$};
	\node[align=center](BIII)[left of=BII]{$B_4$ of $G_{k}$};
	\node[align=center](BIV)[left of=BIII]{$B_2 \cup B_3 \cup B_4$ \\ of $G_{k}$};
	
	\node[draw, fit=(BI)](I){};
	\node[above of=I, node distance=1cm](I-label){$B_1$};
	\node[draw, fit=(BII)](II){};
	\node[above of=II, node distance=1cm](II-label){$B_2$};
	\node[draw, fit=(BIII)](III){};
	\node[above of=III, node distance=1cm](III-label){$B_3$};
	\node[draw, fit=(BIV)](IV){};
	\node[above of=IV, node distance=1cm](IV-label){$B_4$};
	
	\path[->, bend left, blue] (BII) edge (BI);
	\path[->, bend left, blue] (BIII) edge (BII);
	\path[->, bend left, blue] (BIV) edge (BIII);
	
	\path[->, bend right, red] (BIV) edge (BII);
	\path[->, bend right, red] (BIV) edge (BIII);
	\end{tikzpicture}
\end{example}

\paragraph*{Maximum Weight Path}

\begin{theorem}
	The maximum weight path from \rrec{0} to \rrec{I^{k}_t} in $G_k$ for $k \geq 2$ is zero.
\end{theorem}
\begin{proof}
	(Proof by induction)
	
	\begin{itemize}
		\item Base cases: Checking $G_2, G_3$, $G_4$, and $G_5$ in examples \ref{P2:G_2 - G_4} and \ref{P2:G_5} shows that the maximum weight path from \rrec{0} to \rrec{I_{t}^{k}} is $0$ for $k=2,3,4,5$. Also, the ``bottom halves" of $G_4$ and $G_5$, i.e. $B_2 \cup B_3 \cup B_4$, has a maximum weight path of zero from \rrec{IV_{b}^{k}} to \rrec{I_{b}^{k}}. (See examples \ref{P2:G_5} and \ref{P2:G_4 boxes ex}.)
		
		\item Assume for induction that $G_k$ has a maximum weight path of zero and that the maximum weight path from \rrec{IV_{b}^{k}} to \rrec{I_{b}^{k}} is zero for $k \leq N$ where $N \geq 5$. Consider $G_{N+1}$. By Propositions \ref{P2:Box I structure} and \ref{P2:Box II-IV zero edge} there are two choices to travel from \rrec{0}.
		\begin{itemize}
			\item[$\circ$] A path from \rrec{0} to $B_1$ is taken:
			
			Then the path is entirely contained inside $G_N$ by Proposition \ref{P2:Box I structure}, so the maximum weight is zero by the inductive hypothesis.
			
			\item[$\circ$] The path from \rrec{0} to \rrec{IV_{b}^{N+1}} is taken:
			
			Here again there are two options by Proposition \ref{P2:Box I outside edge}.
			\begin{itemize}
				\item[$\bullet$] A path from \rrec{IV_{b}^{N+1}} to a node in $B_4$ is taken:
				
				By Propositions \ref{P2:Box II structure} and \ref{P2:Box III structure}, all such paths must go through \rrec{III_{b}^{N+1}}. By Proposition \ref{P2:Box IV structure}, this portion of the path is identical to the bottom half of $G_N$, so by the inductive hypothesis, this portion has a maximum weight of zero. Again, by Propositions \ref{P2:Box II structure} and \ref{P2:Box III structure}, the path from \rrec{III_b^{N+1}} to \rrec{I_b^{N+1}} is a copy of the bottom half of $G_N$ except with fewer edges of weight $+1$, thus this portion of the path must have a maximum weight of at most zero. Therefore, \rrec{0} to \rrec{IV_{b}^{N+1}} and then the path from \rrec{IV_{b}^{N+1}} to \rrec{I_{b}^{N+1}} has a maximum weight of at most one, which is the same as taking the direct path from \rrec{0} to \rrec{I_{b}^{N+1}}. Since the direct path is completely contained in $G_N$ by Proposition \ref{P2:Box I structure}, it results in a maximum weight path of zero. Replacing the direct path with a longer path through the bottom half with the same weight will still result in a maximum weight path of zero.
				
				\item[$\bullet$] A path from \rrec{IV_{b}^{N+1}} to a node in $B_2$ is taken:
				
				Since \rrec{IV_{b}^{N+1}} $\cup B_2$ of $G_{N+1}$ is a copy of \rrec{IV_{b}^{N}} $\cup B_2 \cup B_3$ of $G_N$ by Proposition \ref{P2:Box II structure}, the maximum weight path is zero by the inductive hypothesis.
				
				\item[$\bullet$] The path from \rrec{IV_{b}^{N+1}} to \rrec{III_{b}^{N+1}} in $B_3$ is taken:
				
				By Proposition \ref{P2:Box III structure}, $B_3$ of $G_{N+1}$ is a copy of $B_4$ of $G_N$. This in turn is a copy of $B_2 \cup B_3 \cup B_4$ of $G_{N-1}$ by Proposition \ref{P2:Box IV structure}. Thus the maximum weight path from \rrec{III_{b}^{N+1}} to \rrec{II_{b}^{N+1}} is zero by the inductive hypothesis. Therefore, taking the edge from \rrec{IV_{b}^{N+1}} to \rrec{III_{b}^{N+1}} and then a path from there to \rrec{II_{b}^{N+1}} has the same weight as taking the edge from \rrec{IV_{b}^{N+1}} to \rrec{II_{b}^{N+1}} directly, so the maximum weight path is again zero by the above case.
			\end{itemize}
		\end{itemize}
		
	\end{itemize}
\end{proof}

\begin{corollary}\label{cor-rk1-prune_2_mw}
	The maximum weight path from \rrec{0} to \rrec{2^{n}-1} in $G_2 \cup G_3 \cup \cdots \cup G_n$ is zero.
\end{corollary}
\begin{proof}
	Note that taking a red edge from \rrec{0} to a node in $G_k$ and then the path from that node to \rrec{I_{t}^{k}} has maximum weight zero by the above theorem while the red edge from \rrec{0} to \rrec{I_{t}^{k}} has a weight of one. Therefore the maximum weight path must take a red edge from \rrec{0} to a node in $G_n$. Thus the maximum weight path lies entirely in $G_n$.
\end{proof}

\begin{example}
	$G_2 \cup G_3 \cup G_4 \cup G_5$.
	
	\vspace*{0.5cm}
	\centering
	\begin{tikzpicture}[> = stealth, shorten > = 1pt, semithick, node distance=1.5cm]
	\tikzstyle{rec} = [draw=black, thick, rectangle, fill=white, align=center]
	\tikzstyle{crc} = [draw=black, thick, fill=white, circle]
	
	\node[crc](31){$31$};
	\node(E5)[left of=31]{$\cdots$};
	\node[crc](15)[left of=E5]{$15$};
	\node(E4)[left of=15]{$\cdots$};
	\node[crc](7)[left of=E4]{$7$};
	\node(E3)[left of=7]{$\cdots$};
	\node[crc](3)[left of=E3]{$3$};
	\node(E2)[left of=3]{$\cdots$};
	\node[crc](1)[left of=E2]{$1$};
	\node[crc](0)[below of=1, node distance=2cm]{$0$};
	
	\node[draw, fit=(15) (31)](G5){};
	\node[above of=G5, node distance=1cm](G5-label){$G_5$};
	\node[draw, fit=(7) (15), dashed](G4){};
	\node[above of=G4, node distance=1cm](G4-label){$G_4$};
	\node[draw, fit=(3) (7)](G3){};
	\node[above of=G3, node distance=1cm](G3-label){$G_3$};
	\node[draw, fit=(1) (3), dashed](G2){};
	\node[above of=G2, node distance=1cm](G2-label){$G_2$};
	
	\path[->, red] (0) edge (1);
	\path[->, bend right, red] (0) edge (E2);
	\path[->, bend right, red] (0) edge (3);
	\path[->, bend right, red] (0) edge (E3);
	\path[->, bend right, red] (0) edge (7);
	\path[->, bend right, red] (0) edge (E4);
	\path[->, bend right, red] (0) edge (15);
	\path[->, bend right, red] (0) edge (E5);
	\end{tikzpicture}
\end{example}

By construction of the graph for $\E_{2}$, Corollary \ref{cor-rk1-prune_2_mw} gives us $\displaystyle\underset{(R,C) \in \E_2(S_2)}{\max}\; \frac{C}{R} = 2$. Thus for any $S \in S_{2}$, the number of cases checked $C$ is at worst $2R$, so at most $\vert S \vert$ extraneous cases are checked.

\bibliographystyle{plain}

{\small
\begin{thebibliography}{99}
	\bibitem{coss2018}
	O.~Coss, H.~Hong, J.~D.~Hauenstein, and D.~K.~Molzahn,
	{\em Locating and Counting Equilibria of the Kuramoto Model with Rank-One Coupling},
	SIAM J. Appl. Algebra Geom, 2 (2018), pp. 45--71.
	
	\bibitem{kuramoto1975}
	Y.~Kuramoto,
	{\em Self-entrainment of a population of coupled non-linear oscillators}, 
	International Symposium on Mathematical Problems in Theoretical Physics: January 23--29, 1975, Kyoto University, Kyoto/Japan, Springer, Berlin, 1975, pp. 420--422.
\end{thebibliography}
}

\end{document}